\documentclass{amsart}
\usepackage{hyperref,pb-diagram,amssymb,epic,eepic,verbatim,graphicx,graphics,epsfig,psfrag}
\hypersetup{colorlinks}
\usepackage{tikz-cd} 

\usepackage{color}

\definecolor{darkred}{rgb}{0.5,0,0}
\definecolor{darkgreen}{rgb}{0,0.5,0}
\definecolor{darkblue}{rgb}{0,0,0.5}

\hypersetup{ colorlinks,
linkcolor=darkblue,
filecolor=darkgreen,
urlcolor=darkred,
citecolor=darkblue }

\newcommand\redout{\bgroup\markoverwith
{\textcolor{red}{\rule[.5ex]{2pt}{0.4pt}}}\ULon}
\usepackage{xcolor}

\theoremstyle{plain}
\newtheorem{theorem}{Theorem}[section]
\newtheorem{lemma}[theorem]{Lemma}

\newtheorem{definition}[theorem]{Definition}

\theoremstyle{remark}
\newtheorem{remark}[theorem]{Remark}

\newcommand\M{\mathcal{M}}

\renewcommand\M{\mathcal{M}}

\newcommand{\XX}{\mathcal{X}}

\newcommand{\R}{\mathbb{R}}

\newcommand{\C}{\mathbb{C}}
\newcommand{\cC}{\mathcal{C}}

\newcommand{\Z}{\mathbb{Z}}
\newcommand{\Q}{\mathbb{Q}}

\renewcommand{\P}{\mathbb{P}}
\newcommand{\bA}{\mathbb{A}}

\newcommand\lie[1]{\mathfrak{#1}}

\newcommand{\g}{\lie{g}}

\newcommand{\q}{\lie{q}}
\renewcommand{\l}{\lie{l}}

\newcommand{\z}{\lie{z}}
\renewcommand{\t}{\lie{t}}

\renewcommand{\u}{\lie{u}}

\newcommand{\on}{\operatorname}

\newcommand{\st}{\on{st}}

\newcommand{\quot}{\on{quot}}

\newcommand{\dual}{\vee}

\newcommand{\Edge}{\on{Edge}}

\newcommand{\Ver}{\on{Vert}}

\newcommand{\Proj}{\on{Proj}}

\newcommand{\Aut}{ \on{Aut} }

\newcommand{\Ad}{ \on{Ad} }

\newcommand{\Hom}{ \on{Hom}}

\newcommand{\Spec}{\on{Spec}}

\newcommand{\ssm}{-}

\newcommand\dirac{/\kern-1.2ex\partial} 
\newcommand\qu{/\kern-.7ex/} 
\newcommand\lqu{\backslash \kern-.7ex \backslash} 

\newcommand\dr{r_+ \kern-.7ex - \kern-.7ex r_-}
 
\newcommand{\lev}{{\on{lev}}} 

\usepackage{amsmath, amsfonts,amssymb, latexsym, mathrsfs}
\newtheorem{example}[theorem]{Example}

\newcommand{\labell}\label

\newcommand{\ol}{\overline}

\newcommand{\lan}{\langle}
\newcommand{\ran}{\rangle}

\newcommand{\ti}{\tilde}

\newcommand\Def{\on{Def}}
\newcommand\age{\on{age}}

\renewcommand{\ss}{\on{ss}}

\newcommand{\us}{\on{us}}

\newcommand\MM{\mathfrak{M}}

\newcommand\Gr{\on{Gr}}

\newcommand\rank{\on{rank}}
\newcommand\ev{\on{ev}}

\newcommand\ul{\underline}
\newcommand\mO{\mathcal{O}}

\newcommand\bra[1]{ < \kern-.7ex {#1} \kern-.7ex >} 
\newcommand\bdefn{\begin{definition}}
\newcommand\edefn{\end{definition}}
\newcommand\bea{\begin{eqnarray*}}
\newcommand\eea{\end{eqnarray*}}
\newcommand\bcv{\left[ \begin{array}{r} }
\newcommand\ecv{\end{array} \right] }

\newcommand\bma{\left[ \begin{array} }
\newcommand\ema{\end{array} \right]}
\newcommand\ben{\begin{enumerate}}
\newcommand\een{\end{enumerate}}
\newcommand\beq{\begin{equation}}
\newcommand\eeq{\end{equation}}
\newcommand\bex{\begin{example}}
\newcommand\bsj{\left\{ \begin{array}{rrr} }
\newcommand\esj{\end{array} \right\}}

\newcommand\cI{\mathcal{I}}

\newcommand\eex{\end{example}}

\newcommand\sx{*\kern-.5ex_X}

\newcommand{\fr}{{\on{fr}}}

\newcommand{\cT}{{\mathcal{T}}}

\def\mathunderaccent#1{\let\theaccent#1\mathpalette\putaccentunder}
\def\putaccentunder#1#2{\oalign{$#1#2$\crcr\hidewidth \vbox
to.2ex{\hbox{$#1\theaccent{}$}\vss}\hidewidth}}

\author{Eduardo Gonz\'alez} 
\thanks{Partially supported by 
DMS-1510518}
\address{
Department of Mathematics,
University of Massachusetts Boston,
100 William T. Morrissey Boulevard,
Boston, MA 02125, U.S.A.}
\email{eduardo@math.umb.edu}

\author{Pablo Solis}
\address{Department of Mathematics,
California Institute of Technology, 1200 East California
Boulevard, Pasadena, CA 91125, U.S.A.}

\email{pablos@caltech.edu}

\author{Chris T. Woodward}
\thanks{Partially supported by 
DMS-1207194}
\address{Mathematics-Hill Center,
Rutgers University, 110 Frelinghuysen Road, Piscataway, NJ 08854-8019,
U.S.A.}  \email{ctw@math.rutgers.edu}

\begin{document}

\title{Stable gauged maps}

\maketitle

\begin{abstract}  
  We give an introduction to moduli stacks of gauged maps satisfying a
  stability condition introduced by Mundet \cite{mund:corr} and
  Schmitt \cite{schmitt:univ}, and the associated integrals giving
  rise to gauged Gromov-Witten invariants.  We survey various
  applications to cohomological and K-theoretic Gromov-Witten
  invariants.
\end{abstract}

\tableofcontents

\section{Introduction}

The moduli stack of maps from a curve to the stack quotient of a
smooth projective variety by the action of a complex reductive group
has a natural stability condition introduced by Mundet in
\cite{mund:corr} and investigated further in Schmitt
\cite{schmitt:univ,schmitt:git}; the condition generalizes stability
for bundles over a curve introduced by Mumford, Narasimhan-Seshadri
and Ramanathan \cite{ra:th}.  Let $X$ be a smooth linearized
projective $G$-variety such that the semi-stable locus is equal to the
stable locus, and let denote $X/G$ the quotient stack.  By definition
a map from a curve $C$ to $X/G$ is a pair that consists of a bundle
$P \to C$ and a section $u$ of the associated bundle
$P \times_G X \to C$.  We denote by $\pi: X/G \to \on{pt}/G =: BG$ the
projection to the classifying space.  In case $X$ is a point, a
stability condition for $\Hom(C,X/G)$, bundles on $C$, was introduced
by Ramanathan \cite{ra:th}.  For $X$ not a point,
a stability condition that combines bundle and target stability was
introduced by Mundet \cite{mund:corr}.  There is a compactified moduli
stack $\ol{\M}^G_n(C,X,d)$ whose open locus consists of Mundet
semistable maps of class $d \in H_2^G(X,\Z)$ with markings
$$ v: C \to X/G, \quad (z_1,\ldots, z_n) \in C^n \
\text{distinct},$$
and where the notion of semi-stability depends on a choice of
linearization \(\ti{X}\to X\).
The compactification uses the notion of Kontsevich stability for
maps \cite{qk1}, \cite{qk2}, \cite{qk3}.  The stack admits evaluation
maps to the quotient stack
$$ \ev: \ol{\M}^G_n(C,X,d) \to (X/G)^n,\quad
(\hat{C},P,u,\ul{z})\mapsto (z_i^* P, u \circ z_i) .$$
In addition, assuming stable=semistable there is a virtual fundamental
class constructed via the machinery of Behrend-Fantechi \cite{bf:in}.
Let $\widehat{QH}_G(X)$ denote the formal completion of $QH_G(X)$ at
$0$.  The {\em gauged Gromov-Witten trace} is the map
$$ \tau_X^G : \widehat{QH}_G(X) \to \Lambda_X^G, \quad \alpha \mapsto
\sum_{n,d} \frac{q^d}{n!}  \int_{\ol{\M}_n^G(C,X,d)} \ev^*
(\alpha,\ldots,\alpha) .$$
Here \(\Lambda_X^G\) is the (equivariant) field of Novikov
variables (see Definition
\ref{d:oqh}.)
The derivatives of the potential will be called \emph{gauged Gromov-Witten
invariants}.  For toric varieties, the potential $\tau_X^G$ already
appears in Givental \cite{gi:eq} and Lian-Liu-Yau \cite{lly:mp1} under
the name of {\em quasimap potential}.\footnote{We are simplifying
  things a bit for the sake of exposition; actually the quasimap
  potentials in those papers involve an additional determinant line
  bundle in the integrals.}  In those papers (following earlier work
of Morrison-Plesser \cite{mp:si}) the gauged potential is explicitly
computed in the toric case, and questions about Gromov-Witten
invariants of toric varieties or complete intersections therein
reduced to a computation of quasimap invariants.  So what we are
concerned with here is the generalization of that story to arbitrary
git quotients, especially the non-abelian and non-Fano cases to cover
situations such as arbitrary toric stacks and quiver moduli.
Especially, we are interested in re-proving and extending the results
of those papers in a uniform and geometric way that extends to quantum
K-theory and non-abelian quotients and does not use any assumption
such as the existence of a torus action with isolated fixed points.
The splitting axiom for the gauged invariants is somewhat different
than the usual splitting axiom in Gromov-Witten theory: the potential
$\tau_X^G$ is a non-commutative version of a {\em trace} on the
Frobenius manifold $QH_G(X)$.  Note that there are several other
notions of gauged Gromov-Witten invariants, for example,
Ciocan-Fontanine-Kim-Maulik \cite{cf:st}, Frenkel-Teleman-Tolland
\cite{toll:gw1}, as well as a growing body of work on gauged
Gromov-Witten theory with potential \cite{tianxu}, \cite{glsm}.

The gauged Gromov-Witten invariants so defined are closely related to,
but different from in general, the graph Gromov-Witten invariants of
the stack-theoretic geometric invariant theory quotient.  The stack of
marked maps to the git quotient
$$ v: C \to X \qu G, \quad (z_1,\ldots, z_n) \in C^n \ \text{distinct} $$
is compactified by the {\em graph space} \label{graphs}
$$ \ol{\M}_n(C, X \qu G, d) := \ol{\M}_{g,n}(C \times X \qu G,
(1,d))$$
the moduli stack of stable maps to $C \times X \qu G$ of class
$(1,d)$; in case $X \qu G$ is an orbifold the domain is allowed to
have orbifold structures at the nodes and markings as in
\cite{cr:orb}, \cite{agv:gw}.  The stack admits evaluation maps
$$ \ev: \ol{\M}_n(C,X \qu G,d) \to (\ol{\cI}_{X \qu G})^n $$
where $\ol{\cI}_{X\qu G}$ is the {\em rigidified inertia stack} of
$X\qu G$.  
The {\em graph trace} is the map
$$ \tau_{X \qu G}: \widehat{QH}_{\C^\times}(X \qu G) \to \Lambda_X^G, \quad \alpha \mapsto
\sum_{n,d} \frac{q^d}{n!}  \int_{\ol{\M}_n(C,X \qu G,d)} \ev^*
(\alpha,\ldots,\alpha) $$
where the equivariant parameters are interpreted as Chern classes of
the cotangent lines at the markings.  The relationship between the
graph Gromov-Witten invariants of $X \qu G$ and Gromov-Witten
invariants arising from stable maps to $X \qu G$ in the toric case is
studied in \cite{gi:eq}, \cite{lly:mp1}, and other papers.

The goal of this paper is to describe, from the point of view of
algebraic geometry, a cobordism between the moduli stack of Mundet
semistable maps and the moduli stack of stable maps to the git
quotient with corrections coming from ``affine gauged maps''.  Affine
gauged maps are data
$$ v: \P^1 \to X /G, \quad v(\infty) \in X^{\ss}/G, \quad
z_1,\ldots,z_n \in \P^1 - \{ \infty \}  \ \text{distinct}$$
where $\infty = [0,1] \in \P^1$ is the point ``at infinity'', modulo
{\em affine} automorphisms, that is, automorphisms of $\P^1$ which
preserve the standard affine structure on $\P^1 - \{ 0 \}$.  Denote by
$\ol{\M}^G_{n,1}(\bA,X)$ the compactified moduli stack of such affine
gauged maps to $X$; we use the notation $\bA$ to emphasize that the
equivalence only uses affine automorphisms of the domains.  
 \label{scaledaffine} Evaluation at the markings defines a morphism
$$ \ev \times \ev_\infty: \ol{\M}^G_{n,1}(\bA,X,d) \to (X/G)^n \times
\ol{\cI}_{X \qu G} .$$
In the case $d = 0$, the moduli stack $\ol{\M}^G_{0,1}(\bA,X,d)$ is
isomorphic to $\ol{\cI}_{X \qu G}$ via evaluation at infinity. 
The {\em quantum Kirwan map} is the map
$$ \kappa_X^G: \widehat{QH}_G(X) \to QH(X \qu G) $$
defined as follows.  Let $\ev_{\infty,d}: \ol{\M}_{n,1}^G(\bA,X,d) \to
\ol{\cI}_{X \qu G}$ be evaluation at infinity restricted to affine
gauged maps 
of class $d$, 
and
$$ \ev_{\infty,d,*}: H( \ol{\M}_{n,1}^G(\bA,X,d)) \otimes_\Q \Lambda_X^G \to
H_G(\ol{\cI}_{X \qu G}) \otimes_\Q \Lambda_X^G $$
push-forward using the virtual fundamental class.  The quantum Kirwan
map is
$$ \kappa_X^G: \widehat{QH}_G(X) \to QH(X \qu G), \quad \alpha \mapsto
\sum_{n,d} \frac{q^d}{n!} \ev_{\infty,d,*} \ev^* (\alpha,\ldots,
\alpha) .$$
As a formal maps, each term in the Taylor series of $\kappa_X^G$ and
$\tau_X^G$ is well-defined on $QH_G(X)$, but in general the sums of
the terms may have convergence issues.  The $q =0 $ specialization of
$\kappa_X^G$ is the Kirwan map to the cohomology of a git quotient
studied in \cite{ki:coh}.

The cobordism relating stable maps to the quotient with Mundet
semistable maps is itself a moduli stack of gauged maps with
extra structure, a scaling, defined by allowing the linearization to tend
towards infinity, that is, by considering Mundet
semistability with respect to the linearization   
\(\ti{X}^k\) as \(k\) goes to infinity.  In
order to determine which stability condition to use, the source curves
must be equipped with additional data of a {\em scaling}: a section
$$ \delta:
\hat{C} \to \P \left(\omega_{\hat{C}/(C \times S)} \oplus
  \mO_{\hat{C}} \right) $$
of the projectivized relative dualizing sheaf.  If the section is
finite, one uses the Mundet semistability condition, while if infinite
one uses the stability condition on the target.  The possibility of
constructing a cobordism in this way was suggested by a symplectic
argument of Gaio-Salamon \cite{ga:gw}.  A {\em scaled gauged map} is a
map to the quotient stack whose domain is a curve equipped with a
section of the projectivized dualizing sheaf and a collection of
distinct markings: A datum
$$ \hat{C} \to S, \quad v:\hat{C} \to C \times X/G, \quad 
\delta: \hat{C} \to \P \left(\omega_{\hat{C}/(C \times S)} \oplus
  \mO_{\hat{C}} \right), \quad z_1,\ldots, z_n \in \hat{C} $$
where 
\begin{itemize}
  \item[--] $\hat{C} \to S$ is a nodal curve of genus $g=\on{genus}{C}$,
  \item[--] $v = (P,u)$ is a morphism to the quotient stack $X/G$ that
    consists of a principal $G$-bundle $P \to \hat{C}$ and a map
    $u: \hat{C} \to P \times_G X$ whose class projects to
    $[C] \in H_2(C)$, so that a sub-curve of $\hat{C}$ is isomorphic
    to $C$ and all other irreducible components map to points in $C$;
    and
\item[--] $\delta$ is a section of the projectivization of the relative
  dualizing sheaf $\omega_{\hat{C}/(C \times S)}$ satisfying certain
  properties.
\end{itemize} 
In the case that $X \qu G$ is an orbifold, the domain $\hat{C}$ is
allowed to have orbifold singularities at the nodes and markings and
the morphism is required to be representable.  In particular, in the
case $X,G$ are points and $n = 0$, the stability condition requires
$\hat{C} \cong C$ and the moduli space $\ol{\M}_{0,1} \cong \P^1$ is
the projectivized space of sections $\delta$ of $\omega_{\hat{C}/C
  \times S} \cong \mO_C$.  The moduli stack of stable scaled gauged
maps $\ol{\M}^G_{n,1}(C,X,d)$ \label{scaledproj} with $n$ markings and
class $d \in H_2^G(X,\Q)$ is equipped with a forgetful map
$$ \rho: \ol{\M}^G_{n,1}(C,X,d) \to \ol{\M}_{0,1} \cong \P^1, \quad
[\hat{C}, u, \delta, \ul{z}] \mapsto \delta .$$
The fibers of $\rho$ over zero $0,\infty \in \P^1$ consist of either
Mundet semistable gauged maps, in the case $\delta = 0$, or stable
maps to the git quotient together with affine gauged maps, in the case
$\delta = \infty$: In notation,
\begin{multline} \label{fibers}
\rho^{-1}(0) = \ol{\M}^G_n(C,X,d), \quad \rho^{-1}(\infty) =
\bigcup_{d_0 + \ldots + d_r = d} \bigcup_{I_1 \cup \ldots \cup I_r
  =\{1,\ldots,n\} } \\ ( \ol{\M}_{g,r}^{\fr}(C \times X \qu G,(1,d_0))
\times_{(\ol{\cI}_{X \qu G})^r} \prod_{j=1}^r \ol{\M}^G_{|I_j|,1}
(\bA,X,d_j))/(\C^\times)^r \end{multline}
where we identify $H_2(X \qu G)$ as a subspace of $H_2^G(X)$ via the
inclusion $X \qu G \subset X / G$, and
$\ol{\M}_{g,r}^{\fr}(C \times X \qu G,(1,d_0))$ denotes the moduli
space of stable maps with framings of the tangent spaces at the
markings.  The properness of these moduli stacks was argued via
symplectic geometry in \cite{qk2}.  We give an algebraic proof in
\cite{reduc}.

The cobordism of the previous paragraph gives rise to a relationship
between the gauged invariants and the invariants of the quotient that
we call the quantum Witten formula.  The formula expresses the failure
of a diagram
%
%
\begin{equation} 
  \label{witten}
\begin{tikzcd}[every arrow/.append style={-latex}]
  \widehat{QH}_G(X) \arrow{dr}[swap]{\tau_X^G}
  \arrow{rr}{{\kappa}_X^G} & &\widehat{QH}( X \qu G)
  \arrow{dl}{\tau_{X \qu G}} \\ 
  & \Lambda_X^G &
   \end{tikzcd}
\end{equation}
to commute as an explicit sum of contributions from wall-crossing
terms.  Here $\widehat{QH}_G(X), \widehat{QH}(X \qu G)$ denote formal
completions of the quantum cohomologies, $\Lambda_X^G$ is the
equivariant Novikov ring and the diagonal arrows are the potentials
that arise from virtual integration over the certain moduli stacks of
gauged maps.  The wall-crossing terms vanish in the limit of large
linearization and are not discussed in this paper.  We explain how
this gives rise to the diagram \eqref{witten}, at least in the large
linearization limit.

\begin{theorem} {\rm  (Adiabatic limit theorem, \cite{qk3})}
The diagram \eqref{witten} commutes in the limit of large
linearization $\ti{X}^k, k \to \infty$, that is, 
$$ \lim_{k \to \infty} \tau_X^G = \tau_{X \qu G} \circ \kappa_X^G
.$$
\end{theorem}

Before giving the proof perhaps we should begin by explaining what
applications we have in mind for gauged Gromov-Witten invariants of
this type and the adiabatic limit theorem in particular.  Many
interesting varieties in algebraic geometry have presentations as git
quotients.  The Grassmannian and projective toric varieties are
obvious examples; well-studied also are quiver varieties such as the
moduli of framed sheaves on the projective plane via the
Atiyah-Drinfeld-Hitchin-Manin construction.  In each of these cases,
the moduli stacks of gauged maps are substantially simpler than the
moduli stacks of maps to the git quotients.  This is because the
``upstairs spaces'' are affine, and so the moduli spaces (at least in
the case without markings) consist simply of a bundle with section
up to equivalence.  In many cases this means that the gauged
Gromov-Witten invariants can be explicitly computed, even though the
Gromov-Witten invariants of the git quotient cannot. Sample
applications of the quantum Kirwan map and the adiabatic limit theorem
include presentations of the quantum cohomology rings of toric
varieties (more generally toric stacks with projective coarse moduli
spaces) \cite{gw:surject} and formulas for quantum differential
equations on, for example, the moduli space of framed sheaves on the
projective plane \cite{cross}.  More broadly, the gauged Gromov-Witten
invariants often have better conceptual properties than the
Gromov-Witten invariants of the git quotients.  So for example, one
obtains from the adiabatic limit theorem a wall-crossing formula for
Gromov-Witten invariants under variation of git, which in particular
shows invariance of the graph potentials in the case of a crepant
wall-crossing \cite{wall}.

\begin{proof}[Proof of Theorem] Consider the degree $d$ contributions
  from $(\alpha,\ldots, \alpha)$ to $\tau_X^G$ and $\tau_{X \qu G}
  \circ \kappa_X^G$.  The former contribution is the integral of
  $\ev^*(\alpha,\ldots, \alpha)$ over $\ol{\M}^G_n(C,X,d)$.  By
  \eqref{fibers}, this integral is equal to the integral of
  $\ev^*(\alpha,\ldots, \alpha)$ over
$$ \bigcup_{d_0 + \ldots + d_r = d} \bigcup_{I_1,\ldots,I_r }
 \left(  \ol{\M}_{0,r}^{\fr}(C \times X \qu G,(1,d_0)) \times_{\ol{\cI}_{X
      \qu G}^r}
  \prod_{j=1}^r \ol{\M}^G_{|I_j|,1} (\bA,X,d_j) \right) / 
 (\C^\times)^r
 .$$
  With $i_j = |I_j|$ the integral of $\ev^*(\alpha,\ldots, \alpha)$
  can be written as the push-forward of $\prod_{j=1}^r {\ev^*
    \alpha^{\otimes i_j}/i_j!} $ under the product of evaluation maps
  $$ \ev_{\infty,d_j,*}: \ol{\M}^G_{i_j,1} (\bA,X,d_j) \to
  H(\ol{\cI}_{X \qu G})^r $$
followed by integration over $\ol{\M}_{0,r}^G(X \qu G,d_0)$.  Taking
into account the number $n!/i_1! \ldots i_r!r!$ of unordered
partitions $I_1,\ldots,I_r$ of the given sizes $i_1,\ldots, i_r$, this
composition is equal to the degree $d$ contribution from
$(\alpha,\ldots, \alpha)$ to $ \tau_{X \qu G} \circ \kappa_X^G$.
\end{proof}

\section{Scaled curves}

Scaled curves are curves with a section of the projectivized dualizing
sheaf incorporated, intended to give complex analogs of spaces
introduced by Stasheff \cite{st:hs} such as the multiplihedron,
cyclohedron etc.  The commutativity of diagrams such as \eqref{witten}
will follow from divisor class relations in the moduli space of scaled
curves, in a way similar to the proof of associativity of the quantum
product via the divisor class relation in the moduli space of stable,
$4$-marked genus $0$ curves.

Recall from Deligne-Mumford \cite{dm:irr} and Behrend-Manin
\cite[Definition 2.1]{bm:gw} the definition of stable and prestable
curves.  A {\em prestable curve} over a scheme $S$ is a flat proper
morphism $\pi: C \to S$ of schemes such that the geometric fibers of
$\pi$ are reduced, connected, one-dimensional and have at most
ordinary double points (nodes) as singularities.  A {\em marked
  prestable curve} over $S$ is a prestable curve $\pi: C \to S$
equipped with a tuple $\ul{z} = (z_1,\ldots,z_n): S \to C^n$ of
distinct non-singular sections.  A {\em morphism} $p : C \to D$ of
prestable curves over $S$ is an $S$-morphism of schemes, such that for
every geometric point $s$ of $S$ we have (a) if $\eta$ is the generic
point of an irreducible component of $D_s$, then the fiber of $p_s$
over $\eta$ is a finite $\eta$-scheme of degree at most one, (b) if
$C'$ is the normalization of an irreducible component of $C_s$, then
$p_s(C')$ is a single point only if $C'$ is rational.  A prestable
curve is {\em stable} if it has finitely many automorphisms.  Denote
by $\ol{\M}_{g,n}$ the proper Deligne-Mumford stack of stable curves
of genus $g$ with $n$ markings \cite{dm:irr}.  The stack
$\ol{\MM}_{g,n}$ of prestable curves of genus $g$ with $n$ markings is
an Artin stack locally of finite type \cite[Proposition 2]{be:gw}.

\begin{figure}[ht]
\begin{picture}(0,0)%
\includegraphics{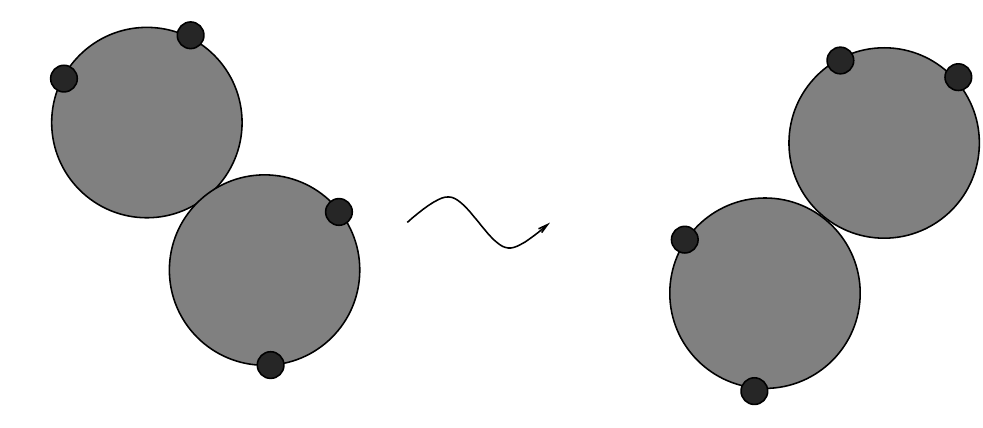}%
\end{picture}%
\setlength{\unitlength}{3947sp}%
\begingroup\makeatletter\ifx\SetFigFont\undefined%
\gdef\SetFigFont#1#2#3#4#5{%
  \reset@font\fontsize{#1}{#2pt}%
  \fontfamily{#3}\fontseries{#4}\fontshape{#5}%
  \selectfont}%
\fi\endgroup%
\begin{picture}(4709,2035)(2021,-1137)
\put(3107,-1080){\makebox(0,0)[lb]{\smash{{\SetFigFont{8}{9.6}{\rmdefault}{\mddefault}{\updefault}{\color[rgb]{0,0,0}$z_0$}%
}}}}
\put(5759,-1086){\makebox(0,0)[lb]{\smash{{\SetFigFont{8}{9.6}{\rmdefault}{\mddefault}{\updefault}{\color[rgb]{0,0,0}$z_0$}%
}}}}
\put(3535, 82){\makebox(0,0)[lb]{\smash{{\SetFigFont{8}{9.6}{\rmdefault}{\mddefault}{\updefault}{\color[rgb]{0,0,0}$z_3$}%
}}}}
\put(6555,682){\makebox(0,0)[lb]{\smash{{\SetFigFont{8}{9.6}{\rmdefault}{\mddefault}{\updefault}{\color[rgb]{0,0,0}$z_3$}%
}}}}
\put(5691,741){\makebox(0,0)[lb]{\smash{{\SetFigFont{8}{9.6}{\rmdefault}{\mddefault}{\updefault}{\color[rgb]{0,0,0}$z_2$}%
}}}}
\put(2036,686){\makebox(0,0)[lb]{\smash{{\SetFigFont{8}{9.6}{\rmdefault}{\mddefault}{\updefault}{\color[rgb]{0,0,0}$z_1$}%
}}}}
\put(3072,787){\makebox(0,0)[lb]{\smash{{\SetFigFont{8}{9.6}{\rmdefault}{\mddefault}{\updefault}{\color[rgb]{0,0,0}$z_2$}%
}}}}
\put(4937,-113){\makebox(0,0)[lb]{\smash{{\SetFigFont{8}{9.6}{\rmdefault}{\mddefault}{\updefault}{\color[rgb]{0,0,0}$z_1$}%
}}}}
\end{picture}%
\caption{Associativity divisor relation} 
\label{assoc} 
\end{figure}
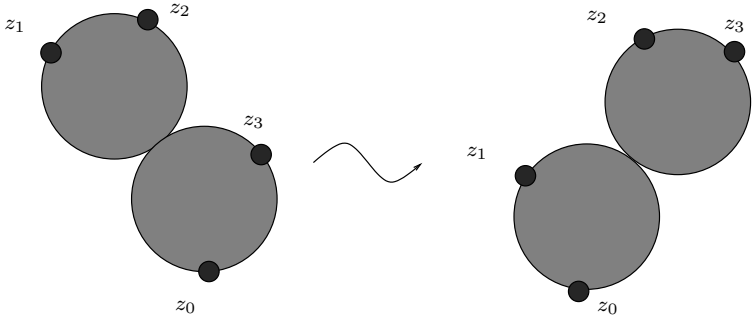 

The following constructions give complex analogs of the spaces
constructed in Stasheff \cite{st:hs}.  For any family of possibly
nodal curves $C \to S$ we denote by $\omega_C$ the relative dualizing
sheaf defined for example in Arbarello-Cornalba-Griffiths
\cite[p. 97]{ar:alg2}.  Similarly for any morphism $\hat{C} \to C$ we
denote by $ \omega_{\hat{C}/C}$ the relative dualizing sheaf and
$\P(\omega_{\hat{C}/C} \oplus \mO_{\hat{C}}) \to \hat{C} $ the
projectivization.  A {\em scaling} is a section
$$ \delta: \hat{C} \to \P(\omega_{\hat{C}/C} \oplus \mO_{\hat{C}}),
\quad \P(\omega_{\hat{C}/C} \oplus \mO_{\hat{C}}) =
(\omega_{\hat{C}/C} \oplus \mO_{\hat{C}})^\times / \C^\times .$$
If $\hat{C} \to C$ is an isomorphism then $\omega_{\hat{C}/C}$ is
trivial:
$$ (\hat{C} \cong C) \implies (\P(\omega_{\hat{C}/C} \oplus
\mO_{\hat{C}}) \cong C \times \P^1) .$$
In this case a scaling $\delta$ is a section $C \to \P^1$, and
$\delta$ is required to be constant.  Thus the space of scalings on an
unmarked, irreducible curve is $\P^1$.

Scalings on nodal curves with markings are required to satisfy the
following properties.  First, $\delta$ should satisfy the {\em
  affinization} property that on any component $\hat{C}_i$ of
$\hat{C}$ on which $\delta$ is finite and non-zero, $\delta$ has no
zeroes.  In particular, this implies that in the case $\hat{C} \cong
C$, then $\delta$ is a constant section as in the last paragraph,
while on any component $\hat{C}_i$ of $\hat{C}$ with finite non-zero
scaling which maps to a point in $C$, $\delta$ has a single double
pole and so defines an affine structure on the complement of the pole.
To define the second property, note that any morphism $\hat{C} \to C$
of class $[C]$ defines a {\em rooted tree} whose vertices are
components $\hat{C}_i$ of $\hat{C}$, whose edges are nodes $w_j \in
\hat{C}$, and whose root vertex is the vertex corresponding to the
component $\hat{C}_0$ that maps isomorphically to $C$.  Let $\cT$
denote the set of indices of {\em terminal} components $\hat{C}_i$
that meet only one other component of $\hat{C}$:
$$ \cT = \{ i \ | \ \# \{ j \neq i | \hat{C}_j \cap \hat{C}_i \neq
\emptyset \} = 1 \} $$
as in Figure \ref{leaves}. The {\em bubble components} are the
components of $\hat{C}$ mapping to a point in $C$.
%
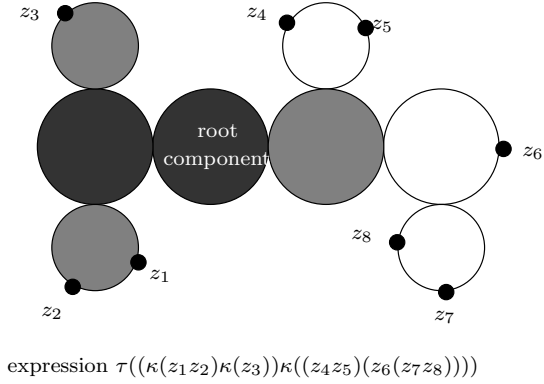
\begin{figure}[ht]
\begin{picture}(0,0)%
\includegraphics{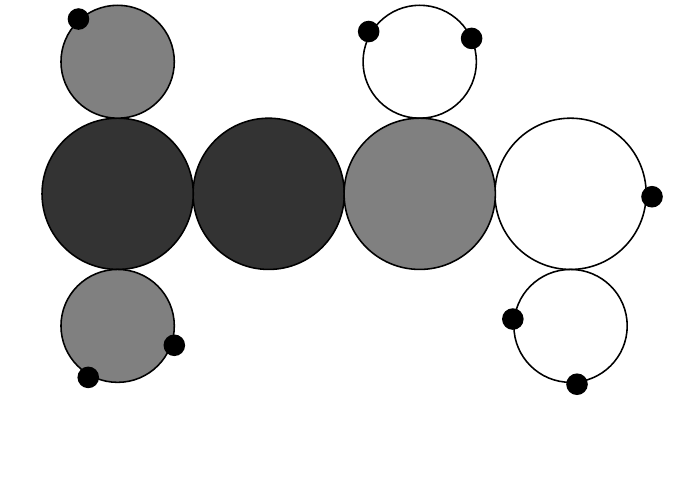}%
\end{picture}%
\setlength{\unitlength}{3947sp}%
\begingroup\makeatletter\ifx\SetFigFont\undefined%
\gdef\SetFigFont#1#2#3#4#5{%
  \reset@font\fontsize{#1}{#2pt}%
  \fontfamily{#3}\fontseries{#4}\fontshape{#5}%
  \selectfont}%
\fi\endgroup%
\begin{picture}(3263,2393)(1695,-4257)
\put(1923,-3849){\makebox(0,0)[lb]{\smash{{\SetFigFont{8}{9.6}{\rmdefault}{\mddefault}{\updefault}{\color[rgb]{0,0,0}$z_2$}%
}}}}
\put(3996,-2049){\makebox(0,0)[lb]{\smash{{\SetFigFont{8}{9.6}{\rmdefault}{\mddefault}{\updefault}{\color[rgb]{0,0,0}$z_5$}%
}}}}
\put(4943,-2849){\makebox(0,0)[lb]{\smash{{\SetFigFont{8}{9.6}{\rmdefault}{\mddefault}{\updefault}{\color[rgb]{0,0,0}$z_6$}%
}}}}
\put(2903,-2727){\makebox(0,0)[lb]{\smash{{\SetFigFont{8}{9.6}{\rmdefault}{\mddefault}{\updefault}{\color[rgb]{1,1,1}root}%
}}}}
\put(2690,-2902){\makebox(0,0)[lb]{\smash{{\SetFigFont{8}{9.6}{\rmdefault}{\mddefault}{\updefault}{\color[rgb]{1,1,1}component}%
}}}}
\put(1710,-4202){\makebox(0,0)[lb]{\smash{{\SetFigFont{8}{9.6}{\rmdefault}{\mddefault}{\updefault}{\color[rgb]{0,0,0}expression $\tau((\kappa(z_1z_2)\kappa(z_3))\kappa( (z_4z_5)(z_6(z_7 z_8))))$}%
}}}}
\put(1784,-1975){\makebox(0,0)[lb]{\smash{{\SetFigFont{8}{9.6}{\rmdefault}{\mddefault}{\updefault}{\color[rgb]{0,0,0}$z_3$}%
}}}}
\put(3210,-1975){\makebox(0,0)[lb]{\smash{{\SetFigFont{8}{9.6}{\rmdefault}{\mddefault}{\updefault}{\color[rgb]{0,0,0}$z_4$}%
}}}}
\put(3870,-3368){\makebox(0,0)[lb]{\smash{{\SetFigFont{8}{9.6}{\rmdefault}{\mddefault}{\updefault}{\color[rgb]{0,0,0}$z_8$}%
}}}}
\put(2602,-3622){\makebox(0,0)[lb]{\smash{{\SetFigFont{8}{9.6}{\rmdefault}{\mddefault}{\updefault}{\color[rgb]{0,0,0}$z_1$}%
}}}}
\put(4389,-3876){\makebox(0,0)[lb]{\smash{{\SetFigFont{8}{9.6}{\rmdefault}{\mddefault}{\updefault}{\color[rgb]{0,0,0}$z_7$}%
}}}}
\end{picture}%
\caption{A scaled marked curve}
\label{leaves}
\end{figure} 
For each terminal component $\hat{C}_i, i \in \cT$ there is a
canonical non-self-crossing path of components $\hat{C}_{i,0} =
\hat{C}_0,\ldots, \hat{C}_{i,k(i)} = \hat{C}_{i}$.  Define a partial
order on components by $\hat{C}_{i,j} \preceq \hat{C}_{i,k}$ for $j
\leq k$.  The {\em monotonicity property} requires that $\delta$ is
finite and non-zero on at most one of these (gray shaded) components,
say $\hat{C}_{i, f(i)}$, and
\begin{equation} \label{monotone} \delta |  \hat{C}_{i,j} = \begin{cases} \infty & j < f(i) \\
                                             0 & j > f(i) \end{cases}
  . \end{equation}
We call $\hat{C}_{i,f(i)}$ a {\em transition component}.  That is, the
scaling $\delta$ is infinite on the components before the transition
components and zero on the components after the transition components,
in the ordering $\preceq$.  See Figure \ref{leaves}.  In addition the
{\em marking property} requires that the scaling is finite at the
markings:
$$\delta(z_i) < \infty,  \quad \forall i =1,\ldots, n .$$

\begin{definition}
 \label{Ccurve}
 A {\em prestable scaled curve} with target a smooth
  projective curve $C$ is a morphism from a prestable curve $\hat{C}$ to
  $C$ of class $[C]$ equipped with a section $\delta$ and $n$ markings
  $\ul{z} = (z_1,\ldots, z_n)$ satisfying the affinization,
  monotonicity and marking properties.  Morphisms of prestable scaled
  curves are diagrams
$$ \begin{diagram} \node{ \hat{C}_1} \arrow{s} \arrow{e,t}{\varphi}
  \node{ \hat{C}_2 } \arrow{s} \\ \node{S_1} \arrow{e}
  \node{S_2} \end{diagram}, \quad (D\varphi^*) \varphi^*( \delta_2) =
\delta_1, \quad \varphi(z_{i,1}) = z_{i,2}, \ \ \forall i = 1,\ldots,
n $$
where the top arrow is a morphism of prestable curves and
$$ D\varphi^* : \varphi^* \P(\omega_{\hat{C}_2/C} \oplus \mO_{\hat{C}_2}) \to
\P(\omega_{\hat{C}_1/C} \oplus \mO_{\hat{C}_1})$$ 
is the associated morphism of projectivized relative dualizing
sheaves.  A scaled curve is {\em stable} if on each bubble component
$\hat{C}_i \subset \hat{C}$ (that is, component mapping to a point in
$C$) there are at least three special points (markings or nodes),
$$ (\delta| \hat{C}_i \in \{0, \infty\} ) \ \implies \ \# (( \{
z_i \} \cup \{ w_j \} ) \cap \hat{C}_i) \ge 3 $$
or the scaling is finite and non-zero and there are least two special points
$$ (\delta| \hat{C}_i \notin \{0, \infty\} ) \ \implies \ \# (( \{
z_i \} \cup \{ w_j \} ) \cap \hat{C}_i) \ge 2 .$$
\end{definition} 

Introduce the following notation for moduli spaces.  Let
$\ol{\MM}_{n,1}(C)$ denote the category of prestable $n$-marked scaled
curves and $\ol{\M}_{n,1}(C)$ the subcategory of stable $n$-marked
scaled curves.

The {\em combinatorial type} of a prestable marked scaled curve is
defined as follows.  Given such $(\hat{C},u: \hat{C} \to C,
\ul{z},\delta)$. Let $\Gamma$ be the graph whose vertex set
$\Ver(\Gamma)$ is the set of irreducible components of $C$, finite
edges $\Edge_{< \infty}(\Gamma)$ correspond to nodes, semi-infinite
edges $\Edge_\infty(\Gamma)$ correspond to markings.
The graph \(\Gamma\) equipped with
the labelling of semi-infinite edges by $\{ 1,\ldots , n\}$ a
distinguished {\em root vertex} $v_0 \in \Ver(\Gamma)$ corresponding
to the root component and a set of {\em transition vertices}
$\Ver^t(\Gamma) \subset \Ver(\Gamma)$ corresponding to the transition
components.  Graphically we represent a combinatorial type as a graph
with transition vertices shaded by grey, and the vertices lying on
three levels depending on whether they occur before or after the
transition vertices.  See Figure \ref{spider}.  Note that the
combinatorial type is functorial; in particular any automorphism of
prestable marked scaled curves induces an automorphism of the
corresponding type, that is, an automorphism of the graph preserving
the additional data.

We note that the graphical representation of the
combinatorial type of a curve can be viewed as the graph of
a Morse/height function on the curve. In general this gives
a spider like figure with the principal component being the
body of the spider. From this perspective the paths used in
the monotonicity property of scalings are the legs of the
spider. 

\begin{figure}[ht]
\includegraphics[height=2in]{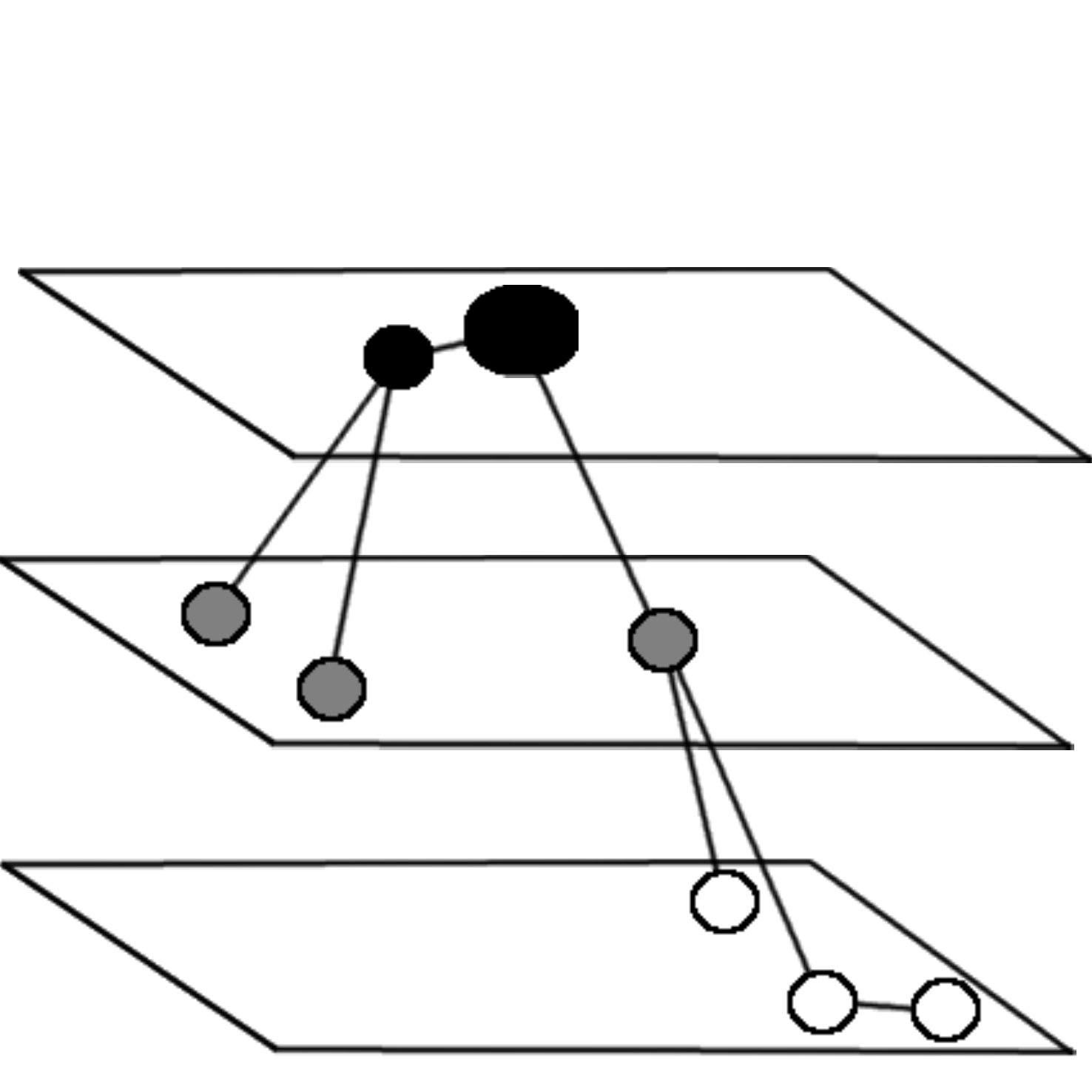}
\caption{Combinatorial type of a scaled marked curve} 
\label{spider}
\end{figure}

\begin{example}
\begin{enumerate}
\item For $n = 0$, no bubbling is possible and $\ol{\M}_{0,1}(C)$ is
  the projective line, $ \ol{\M}_{0,1}(C) \cong \P^1 .$
\item For $n = 1$, $\ol{\M}_{1,1}(C)$ consists of configurations
  $\M_{1,1}(C) \cong C \times \C$ with irreducible domain and finite
  scaling;  configurations $\ol{\M}_{1,1} \ssm \M_{1,1}$ with one
  component $\hat{C}_0 \cong C$ with infinite scaling $\delta |
  \hat{C}_0$, and another component $\hat{C}_1$ mapping trivially to
  $C$, equipped with a one-form $\delta | \hat{C}_1$ with a double
  pole at the node and a marking $z_1 \in \hat{C}_1$.  Thus $
  \ol{\M}_{1,1}(C) \cong C \times \P^1 .$
\item For $n= 2$, $\ol{\M}_{2,1}(C)$ consists of configurations
  $\M_{2,1}(C)$ with two distinct points $z_1, z_2 \in C$ and a
  scaling $\delta \in \P^1$; configurations $\M_{2,1,\Gamma_1}$ where
  the two points $z_1,z_2$ have come together and bubbled off onto a
  curve $z_1,z_2 \in \hat{C}_1$ with zero scaling
  $\delta | \hat{C}_1$, so that
  $\M_{2,1,\Gamma_1} \cong C \times \P^1$; configurations
  $\M_{2,1,\Gamma_2}$ with a root component $\hat{C}_0$ with infinite
  scaling $\delta | \hat{C}_0$, and two components
  $\hat{C}_1,\hat{C}_2$ with non-trivial scalings
  $\delta | \hat{C}_1, \delta | \hat{C}_2$ containing markings
  $z_1 \in \hat{C}_1, z_2 \in \hat{C}_2$; a stratum
  $\M_{2,1,\Gamma_3}$ of configurations with a component $\hat{C}_1$
  containing two markings $z_1,z_2 \in \hat{C}_1$ and
  $\delta | \hat{C}_1$ non-zero; a stratum $\M_{2,1,\Gamma_4}$
  containing with three components, one $\hat{C}_0$ mapping
  isomorphically to $C$; one $\hat{C}_1$ with two nodes and a one form
  $\delta | \hat{C}_1$ with a double pole at the node attaching to
  $\hat{C}_0$; and a component $\hat{C}_2$ with two markings
  $z_1,z_2 \in \hat{C}_2$, a node, and vanishing scaling
  $\delta | \hat{C}_2$; and a stratum $\M_{2,1,\Gamma_5}$ containing
  the root component $\hat{C}_0$, a component $\hat{C}_1$ with
  infinite scaling with three nodes, and two components
  $\hat{C}_2, \hat{C}_3$ with finite, non-zero scaling, each
  containing a node and a marking.  
\end{enumerate} 
\end{example}

\begin{remark} The extension of the one-form in a family of scaled
  curves may be explicitly described as follows.  On each component of
  the limit, the one-form is determined by the limiting behavior of
  the product of deformation parameters for the nodes connecting that
  component to the root component of the limit.  Let
$$\hat{C} \to S,\delta : \hat{C} \to
\P(\omega_{\hat{C}/C\times S} \oplus
  \mO_{\hat{C}}),\ul{z}: S \to \hat{C}^n$$
  be a family of scaled curves over a punctured curve $S = \ol{S} - \{
  \infty \} $ and $\hat{C}_\infty$ a curve over $\infty$ extending the family $\hat{C}$.
Let $\Def(\hat{C}_\infty)/\Def_\Gamma(\hat{C}_\infty)$
  denote the deformation space of the curve $\hat{C}_\infty$ normal to
  the stratum of curves of the same combinatorial type $\Gamma$ as
  $\hat{C}_\infty$.  This normal deformation space is canonically
  identified with the sum of products of cotangent lines at the nodes
$$ \Def(\hat{C}_\infty)/\Def_\Gamma(\hat{C}_\infty) =  \sum_{w} 
T^\dual_w \hat{C}_{i_-(w)} \otimes T^\dual_w \hat{C}_{i_+(w)} $$
where $\hat{C}_{i_\pm(w)}$ are components of $\hat{C}_\infty $
adjacent to $w$, see \cite[p. 176]{ar:alg2}.  Over the deformation
space $\Def(\hat{C}_\infty)$ lives a semiversal family, universal if
the curve is stable.  Given family of curves $\hat{C} \to S$ as above
the curve $\hat{C}$ is obtained by pull-back of the semiversal family
by a map
$$ S  \to \sum_{w} T^\dual_w \hat{C}_{i_-(w)} \otimes
T^\dual_w \hat{C}_{i_+(w)}, \quad z \mapsto (\delta_w(z)) $$
describing the curves as local deformations (non-uniquely, since the
curves themselves may be only prestable.)  Let
$$ \hat{C}_0 = \hat{C}_{i,0}, \ldots, \hat{C}_{i,l(i)} := \hat{C}_i $$
denote the path of components from the root component, and
$$ w_{i,0},\ldots, w_{i,l(i)- 1} \in \hat{C}_\infty $$
the corresponding sequence of nodes.  The nodes $w_{i,j}, w_{i,j+1}$
lie 
in the same component $C_{i,j+1}$ and we
have a canonical isomorphism
$$ T_{w_{i,j}}^\dual C_{i,j+1} \cong T_{w_{i,j+1}}
C_{i,j+1} $$
corresponding to the relation of local coordinates $z_+ = 1/z_-$ near
$w_{i,j}$.  Deformation parameters for this chain lie in the space
%
\begin{multline} 
  \Hom(T_{w_{i,0}}^{\dual} \hat{C}_{i,0}, T_{w_{i,1}}^{\dual} \hat{C}_{i,1})
  \oplus \Hom(T_{w_{i,1}}^{\dual} \hat{C}_{i,1}, T_{w_{i,2}}^{\dual}
  \hat{C}_{i,2}) \ldots \\ \oplus \Hom(T_{w_{i,l(i) - 1}}^{\dual}
  \hat{C}_{i,l(i)-1}, T_{w_{i,l(i)}}^{\dual} \hat{C}_{i,l(i)})
    .\end{multline}
In particular, the product of deformation parameters
\begin{equation} 
\label{product}
\gamma_{w_{i,0}}(z) \cdots \ldots \cdot \gamma_{w_{i,l(i) -1}}(z)
\in \Hom(T_{w_0}^{\dual} \hat{C}_{i,0}, T_{w_i}^{\dual} \hat{C}_{i,l(i)}) \end{equation}
is well-defined.  The product represents the {\em scale} at which the
bubble component $\hat{C}_i$ forms in comparison with $\hat{C}_0 =
\hat{C}_{i,0}$, that is, the ratio between the derivatives of local
coordinates on $\hat{C}_i$ and $\hat{C}_0$.  If $z$ is a point in
$\hat{C}_i$ then we also have a canonical isomorphism
$ T_z^\dual \hat{C}_i \to T_{w_{i,0}} \hat{C}_i .$
The product \eqref{product} gives an isomorphism
$ T_z^\dual \hat{C}_i \to T_{w_0}^{\dual} \hat{C}_{0} .$
The extension of $\delta$ over $\hat{C}_i$ is given by
\begin{equation} \label{oneform} \delta | \hat{C}_i = \lim_{z \to 0}
  \delta(z) (\gamma_{w_{i,0}}(z) \cdots \ldots \cdot \gamma_{w_{i,l(i)
      -1}}(z)) \end{equation}
the ratio of the scale of the bubble component with the 
parameter $\delta(z)^{-1}$.  This ends the Remark. 
\end{remark}

One may view a scaled curve with infinite scaling on the root
component as a nodal curve formed from the root component and a
collection of bubble trees as follows.  

\begin{definition} \label{affine} An {\em affine prestable
  scaled curve} consists of a tuple $(C,\delta,\ul{z})$ where $C$ is a
  connected projective nodal curve, $\delta: C \to \P( \omega_C \oplus
  \mO_C)$ a section of the projectivized dualizing sheaf, and $\ul{z}
  = (z_0,\ldots,z_n)$ non-singular, distinct points, such that
\begin{enumerate} 
\item $\delta$ is monotone in the following sense: For each terminal
  component $\hat{C}_{i}, i \in \cT$ there is a canonical
  non-self-crossing path of components
$$\hat{C}_{l(i),0} = \hat{C}_0,\ldots,
  \hat{C}_{i,k(i)} = \hat{C}_i .$$  
For any such non-self-crossing path of
components starting with a root component, that $\delta$ is finite and
non-zero on at most one of these {\em transition components}, say
$\hat{C}_{i, f(i)}$, and the scaling is infinite for all components
before the transition component and zero for components after the
transition component:
$$ \delta | \hat{C}_{i,j} = \begin{cases} \infty & j < f(i) \\ 0 & j
  > f(i) \end{cases} . $$
\item $\delta$ is infinite at $z_0$, and finite at $z_1,\ldots, z_n$.
\end{enumerate} 
A prestable affine scaled curve is {\em stable} if it has finitely
many automorphisms, or equivalently, if each component ${C}_i \subset
{C}$ has at least three special points (markings or nodes),
$$ (\delta| {C}_i \in \{0, \infty\} ) \ \implies \ \# (( \{
z_i \} \cup \{ w_j \} ) \cap {C}_i) \ge 3 $$
or the scaling is finite and non-zero and there are least two special points
$$ (\delta| {C}_i \notin \{0, \infty\} ) \ \implies \ \# (( \{
z_i \} \cup \{ w_j \} ) \cap {C}_i) \ge 2 .$$
\end{definition} 

We will see below in Theorem \ref{scaledproper} that scaled marked
curves have no automorphisms.  Examples of stable affine scaled curves
are shown in Figure \ref{affine}.  Denote the moduli stack of
prestable affine scaled curves resp. stable affine $n$-marked scaled
curves by $\ol{\MM}_{n,1}(\bA)$ resp. $\ol{\M}_{n,1}(\bA)$.

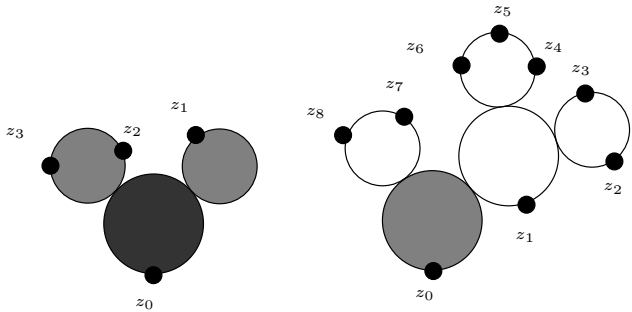
\begin{figure}[ht]
\begin{picture}(0,0)%
\includegraphics{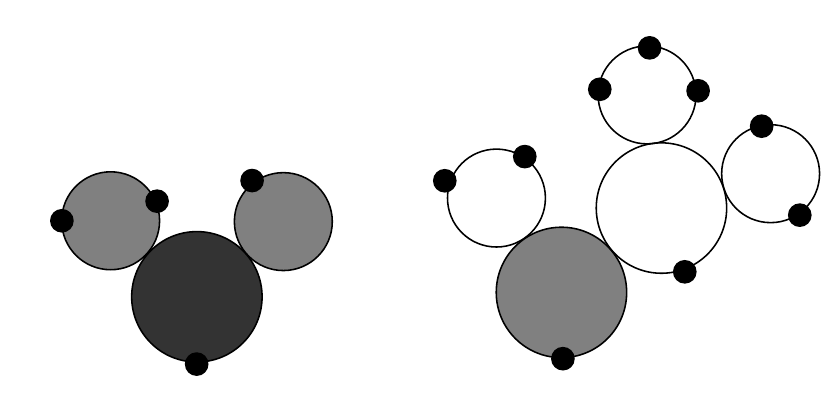}%
\end{picture}%
\setlength{\unitlength}{3947sp}%
\begingroup\makeatletter\ifx\SetFigFont\undefined%
\gdef\SetFigFont#1#2#3#4#5{%
  \reset@font\fontsize{#1}{#2pt}%
  \fontfamily{#3}\fontseries{#4}\fontshape{#5}%
  \selectfont}%
\fi\endgroup%
\begin{picture}(3942,1975)(1401,-3013)
\put(2230,-2970){\makebox(0,0)[lb]{\smash{{\SetFigFont{6}{7.2}{\rmdefault}{\mddefault}{\updefault}{\color[rgb]{0,0,0}$z_0$}%
}}}}
\put(3988,-2913){\makebox(0,0)[lb]{\smash{{\SetFigFont{6}{7.2}{\rmdefault}{\mddefault}{\updefault}{\color[rgb]{0,0,0}$z_0$}%
}}}}
\put(4619,-2548){\makebox(0,0)[lb]{\smash{{\SetFigFont{6}{7.2}{\rmdefault}{\mddefault}{\updefault}{\color[rgb]{0,0,0}$z_1$}%
}}}}
\put(5172,-2245){\makebox(0,0)[lb]{\smash{{\SetFigFont{6}{7.2}{\rmdefault}{\mddefault}{\updefault}{\color[rgb]{0,0,0}$z_2$}%
}}}}
\put(4969,-1500){\makebox(0,0)[lb]{\smash{{\SetFigFont{6}{7.2}{\rmdefault}{\mddefault}{\updefault}{\color[rgb]{0,0,0}$z_3$}%
}}}}
\put(4796,-1354){\makebox(0,0)[lb]{\smash{{\SetFigFont{6}{7.2}{\rmdefault}{\mddefault}{\updefault}{\color[rgb]{0,0,0}$z_4$}%
}}}}
\put(4478,-1129){\makebox(0,0)[lb]{\smash{{\SetFigFont{6}{7.2}{\rmdefault}{\mddefault}{\updefault}{\color[rgb]{0,0,0}$z_5$}%
}}}}
\put(3930,-1359){\makebox(0,0)[lb]{\smash{{\SetFigFont{6}{7.2}{\rmdefault}{\mddefault}{\updefault}{\color[rgb]{0,0,0}$z_6$}%
}}}}
\put(3800,-1599){\makebox(0,0)[lb]{\smash{{\SetFigFont{6}{7.2}{\rmdefault}{\mddefault}{\updefault}{\color[rgb]{0,0,0}$z_7$}%
}}}}
\put(3304,-1771){\makebox(0,0)[lb]{\smash{{\SetFigFont{6}{7.2}{\rmdefault}{\mddefault}{\updefault}{\color[rgb]{0,0,0}$z_8$}%
}}}}
\put(2449,-1734){\makebox(0,0)[lb]{\smash{{\SetFigFont{6}{7.2}{\rmdefault}{\mddefault}{\updefault}{\color[rgb]{0,0,0}$z_1$}%
}}}}
\put(2152,-1886){\makebox(0,0)[lb]{\smash{{\SetFigFont{6}{7.2}{\rmdefault}{\mddefault}{\updefault}{\color[rgb]{0,0,0}$z_2$}%
}}}}
\put(1416,-1896){\makebox(0,0)[lb]{\smash{{\SetFigFont{6}{7.2}{\rmdefault}{\mddefault}{\updefault}{\color[rgb]{0,0,0}$z_3$}%
}}}}
\end{picture}%

\caption{Examples of stable affine scaled curves}
\label{affine}
\end{figure}

\begin{theorem} \label{scaledproper} For each $n \ge 0$ and smooth
  projective curve $C$ the moduli stack $\ol{\M}_{n,1}(C)$ 
   resp. $\ol{\M}_{n,1}(\bA)$ of stable scaled
  affine curves is a proper scheme locally
  isomorphic to a product of a number of copies of $C$ with a toric
  variety.  The stack $\ol{\MM}_{n,1}(C)$ resp. $\ol{\MM}_{n,1}(\bA)$
  of prestable scaled curves is an Artin stack of locally finite type.
\end{theorem}

\begin{proof}  
 Standard arguments on imply that $\ol{\M}_{n,1}(C)$ and
 $\ol{\MM}_{n,1}(C)$ are stacks, that is, categories fibered in
 groupoids satisfying effective descent for objects and for which
 morphisms form a sheaf.  An object $(\hat{C},\ul{z},\delta)$ of
 $\ol{\M}_{n,1}(C)$ over a scheme $S$ is a family of curves with
 sections.  Families of curves with markings and sections satisfy the
 gluing axioms for objects; similarly morphisms are determined
 uniquely by their pull-back under a covering.  Standard results on
 hom-schemes imply that the diagonal for $\ol{\MM}_{n,1}(C)$, hence
 also $\ol{\M}_{n,1}(C)$, is representable, see for example
 \cite[1.11]{dm:irr} for similar arguments, hence the stacks
 $\ol{\MM}_{n,1}(C)$ and $\ol{\M}_{n,1}(C)$ are algebraic.

 In preparation for showing that $\ol{\M}_{n,1}(C)$ is a variety we
 claim that for any object $(\hat{C},\ul{z},\delta)$ of the moduli
 stack $\ol{\M}_{n,1}(C)$ the automorphism group is trivial.  Let
 $\Gamma$ be the combinatorial type.  The association of $\Gamma$ to
 $(\hat{C},\ul{z},\delta)$ is functorial and any automorphism of
 $(\hat{C},\ul{z},\delta)$ induces an automorphism of $\Gamma$.  
 Since the graph $\Gamma$ is a tree with labelled semi-infinite edges,
 each vertex is determined uniquely by the partition of semi-infinite
 edges given by removing the vertex.  Hence the automorphism acts trivially
 on the vertices of $\Gamma$.  Each component has at least three
 special points, or two special points and a non-trivial scaling and
 so has trivial automorphism group fixing the special points.  Thus
 the automorphism is trivial on each component of $\hat{C}$.  The
 claim follows.

 The moduli space of stable scaled curves has a canonical covering by
 varieties corresponding to the versal deformations of prestable
 curves constructed by gluing.  Suppose that $(u: \hat{C} \to C,
 \ul{z}, \delta)$ is an object of $\ol{\M}_{n,1}(C)$ of combinatorial
 type $\Gamma$.  
 Let $\rho: \ol{\M}_{n,1}(C) \to \ol{\M}_{0,1}(C) \cong \P$ denote the
 forgetful morphism. 
 The locus $\rho^{-1}(\C) \subset \ol{\M}_{n,1}(C)$ of
 curves with finite scaling is isomorphic to $\ol{\M}_n(C) \times \C$,
 where the last factor denotes the scaling.  In the case that the
 root component has infinite scaling, let
 $\Gamma_1,\ldots,\Gamma_k$ denote the (possibly empty) combinatorial
 types of the bubble trees attached at the special points.  The
 stratum ${\M}_{n,1,\Gamma}(C)$ is the product of $C^k$ with moduli
 stacks of scaled affine curves ${\M}_{n_i,1,\Gamma_i}(\bA)$ for $i
 =1,\ldots, k$, each isomorphic to an affine space given by the number
 of markings and scalings minus the dimension of the automorphism
 group $(n_i + 1) + 1 - \dim(\Aut(\P^1)) = n_i - 1$ \cite{mau:mult}.
 Let
$$\gamma_e \in T^\dual_{w(e)} \hat{C}_{i_-(e)} \otimes T^\dual_{w(e)}
\hat{C}_{i_+(e)}, \quad e \in \Edge_{< \infty}(\Gamma)$$
be the deformation parameters for the nodes.  A collection of
deformation parameters $ \ul{\gamma} = ( \gamma_e)_{e \in
  \Edge(\Gamma)}$ is {\em balanced} if the signed product
\begin{equation} \label{balanced} \prod_{e \in P} \gamma_e^{\pm 1} \end{equation} 
of parameters corresponding to any non-self-crossing path $P$ between
transition components is equal to $1$, where the sign is positive for
edges pointing towards the root vertex and equal to $-1$ if the edge
is oriented away from it.  Let $Z_\Gamma$ denote the set of
deformation parameters satisfying the condition \eqref{balanced}.
Then there is a morphism
$$ \M_{n,1,\Gamma}(C) \times Z_\Gamma \to \ol{\M}_{n,1}(C) $$
described as follows.  Choose local \'etale coordinates $z_e^\pm$ on
the adjacent components to each node $w_e, \in \Edge_{<
  \infty}(\Gamma)$ and glue together the components using the
identifications $z_e^+ \mapsto \gamma_e/ z_e^-$, see for example
\cite[p. 176]{ar:alg2}, \cite[2.2]{ol:logtwist}.  Set the scaling on
the root component
$$ \delta = \prod_{e \in P} \gamma_e $$
where $P$ is a path of nodes from the root component to the transition
component, independent of the choice of component by \eqref{balanced}.
This gives a family $(\hat{C},u,\delta,\ul{z})$ of stable scaled
curves over $\M_{n,1,\Gamma}(C) \times Z_\Gamma$ and hence a morphism
to $\ol{\M}_{n,1}(C)$.  The family $(\hat{C},\ul{z},u,\delta)$ defines
a universal deformation of any curve of type $\Gamma$.  Indeed,
$(\hat{C},\ul{z})$ is a versal deformation of any of its prestable
fibers by \cite{ar:alg2}, and it follows that the family
$(\hat{C},\ul{z},u)$ is a versal deformation of any of its fibers
since there is a unique extension of the stable map on the central
fiber, up to automorphism.  The equation \eqref{product} implies that
any family of stable scaled curves satisfies the balanced relation
\eqref{balanced} between the deformation parameters for any family of
marked curves with scalings.  This provides a cover of
$\ol{\M}_{n,1}(C)$ by varieties.  It follows that $\ol{\M}_{n,1}(C)$
is a variety.

The stack of prestable scaled curves $\ol{\MM}_{n,1}(C)$ is an Artin
stack of locally finite type.  Charts for the stack
$\ol{\MM}_{n,1}(C)$, as in the case of prestable curves in
\cite{be:gw}, are given by using forgetful morphisms
$\ol{\M}_{n+k,1}(C) \to \ol{\MM}_{n,1}(C)$.  Since these morphisms
admit sections locally, they provide a smooth covering of
$\ol{\MM}_{n,1}(C)$ by varieties.

We check the valuative criterion for properness for
$\ol{\M}_{n,1}(C)$.  Given a family of stable scaled marked curves
over a punctured curve $S$ with finite scaling $\delta$
  $$(\hat{C}, u: \hat{C} \to C, \ul{z},\delta) \to S = \ol{S} - \{
\infty \} $$
we wish to construct there exists an extension over $\ol{S}$.  We
consider only the case $\hat{C} \cong C \times S$; the general case is
similar.  After forgetting the scaling $\delta$ and stabilizing we
obtain a family of stable maps to $C$ of degree $[C]$,
$$ (\hat{C}^{\st}, u: \hat{C}^{\st} \to C, \ul{z}^{\st}) \to 
\ol{S} - \{ \infty \} .$$
By properness of the stack $\ol{\M}_n(C)$ of stable maps to $C$, this
family extends over the central fiber $\infty$ to give a family over
$\ol{S}$.  The section $\delta$ of $\omega_{\hat{C}^{\st}/C}$ defines
an extension over $\ol{S}$ except possibly at the nodes.  Here there
are possible irremovable singularities corresponding to the following
situation: suppose that $\hat{C}_0,\ldots \hat{C}_i$ is a chain of
components in the curve at the central fiber, with $\hat{C}_0 \cong C$
the root component.  Suppose that $\hat{C}_{i}, \hat{C}_{i+1}$
are adjacent component with $\delta$ infinite on $\hat{C}_i$ and zero
on $\hat{C}_{i+1}$.  Taking the closure of the graph of $\delta$ gives
a family $\hat{C}$ of curves over $C$ given by replacing some of the
nodes of $\hat{C}^{\st}$ with fibers of $\P(\omega_{\hat{C}^{\st}/C}
\oplus \mO_{\hat{C}^{\st}})$ over the node.  The relative cotangent
bundle of $\hat{C}$ is related to that of $\hat{C}^{\st}$ by a twist
at $D_0,D_\infty$: If $\pi: \hat{C} \to \hat{C}^{\st}$ denotes the
projection onto $\hat{C}$ then on the components of $\hat{C}$
collapsed by $\pi$ we have
$$ \omega_{\hat{C}/C\times S} = \pi^* \omega_{\hat{C}^{\st}/C} (-D_0 
-D_\infty) $$
where $D_0,D_\infty$ are the inverse images of the sections at zero
and infinity in $\P( \omega_{\hat{C}^{\st}/C} \oplus
\mO_{\hat{C}^{\st}})$.  Abusing notation $\omega_{\hat{C}_i^{\st}/C}
(-D_0 ) = \omega_{\hat{C}_i^{\st}/C}$ resp.
$\omega_{\hat{C}_i^{\st}/C} (-D_\infty ) = \omega_{\hat{C}_i^{\st}/C}$
on components $\hat{C}_i^{\st}$ contained in $D_0$ resp. $D_\infty$.
The extension of $\delta$ to a rational section of $\pi^*
\omega_{\hat{C}^{\st}/C}$ has, by definition a zero at
$\delta^{-1}(D_0)$ and a pole at $\delta^{-1}(D_\infty)$.  Hence the
extension of $\delta$ to a section of $\pi^*
\omega_{\hat{C}^{\st}/C}(-D_0 - D_\infty)$ has no zeroes at $D_0$ and
a double pole at $D_\infty$.  This implies that $\delta$ extends
uniquely as a section of $\P( \omega_{\hat{C}/C\times S} \oplus
\mO_{\hat{C}})$ to all of $\ol{S}$.

By the construction \eqref{oneform}, the extension of $\delta$
satisfies the monotonicity condition \eqref{monotone}.  Indeed suppose
that a component $\hat{C}_i$ is further away from a component
$\hat{C}_j$ in the path of components from the root component
$\hat{C}_0$.  Since all deformation parameters $\gamma_{w_{i,k}}(z)$ are
approaching zero, from \eqref{oneform}, at most one of the limits
$\delta | \hat{C}_i, \delta | \hat{C}_j$ can be finite, and
$$ \begin{cases} \delta | \hat{C}_i   \ \text{finite}   \
  \implies \delta | \hat{C}_j \ \ \text{infinite}  \\
                 \delta | \hat{C}_j   \ \text{finite}   \
				 \implies \delta | \hat{C}_i \ \ \text{zero} .
\end{cases}. $$ 
The condition \eqref{monotone} follows.
\end{proof}

\begin{remark} 
The basic divisor equivalences used in the proof of the cobordism
\eqref{fibers} are the following.  Recall the proof of associativity
of the quantum product via divisor relations in $\ol{\M}_{0,n+1}$.
The moduli stack of genus $0$ curves with $n+1$ markings
$$ \ol{\M}_{0,n+1} =\bigcup_\Gamma \M_{0,n+1,\Gamma} $$
admits a stratification into strata indexed by parenthesized
expressions in $n$ variables $a_1,\ldots,a_n$ in a commutative algebra
$A$. The basic divisor class relation implying associativity in
$\ol{\M}_{0,4}$ corresponds to an expression
$ (a_1 a_2) a_3 \sim a_1 (a_2 a_3) $
corresponding to the picture in Figure \ref{assoc}.   This relation pulls back 
to a relation 
\begin{multline} \left[ \cup_{d_1 + d_2 = d}
\cup_{ 1,2 \in I_1, 0,3 \in I_2} \ol{\M}_{0,I_1}(X,d_1) \times_X
\ol{\M}_{0,I_2}(X,d_2) \right]
\\  \sim \left[ \cup_{d_1 + d_2 = d} \cup_{ 2,3 \in I_1,
  0,1 \in I_2} \ol{\M}_{0,I_1}(X,d_1) \times_X
\ol{\M}_{0,I_2}(X,d_2) \right] \end{multline}
(where $I_1 \cup I_2$ is a partition of $\{0,\ldots, n\}$) in the
moduli stack of stable maps $\ol{\M}_{0,n+1}(X,d)$.  A standard
argument using this relation implies the associativity of the quantum
product \cite[Theorem 6]{be:gw}.

The divisor class relations imply the commutativity of the diagram
\eqref{witten} as follows.  In the moduli space of scaled curves there
is a stratification
$$ \ol{\M}_{n,1}(C) = \bigcup_\Gamma {\M}_{n,1,\Gamma}(C) $$
in which the strata correspond to parenthesized expressions in symbols
$\tau,\kappa$ and $a_1,\ldots, a_n$, such as $\tau(\kappa(a_1 a_2)
\kappa(a_3))$ etc \cite{qk1}.  In particular, there is a divisor class
relation in $\ol{\M}_{0,1}(C)$ corresponding to the relation
$ (\tau \kappa)(a) \sim \tau(\kappa(a)) $
which pulls back to a divisor class relation in $\ol{\M}_{n,1}(C)$.
The trivialization of the relative canonical sheaf defines a canonical
isomorphism
$\ol{\M}_{0,1}(C)\cong\P^1  .$
The quantum Witten formula is obtained from the equivalence between
the divisors corresponding to the configurations in Figure
\ref{config}.

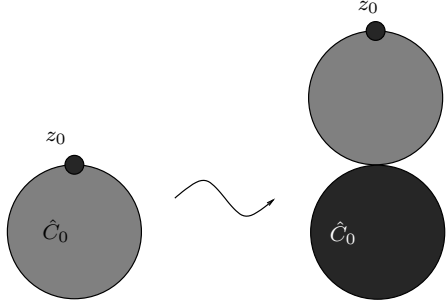
\begin{figure}[ht]
\begin{picture}(0,0)%
\includegraphics{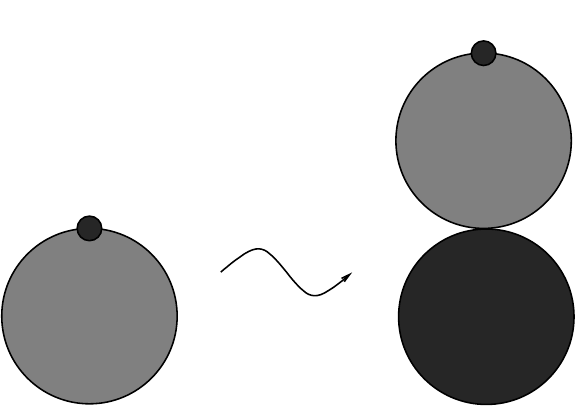}%
\end{picture}%
\setlength{\unitlength}{3947sp}%
\begingroup\makeatletter\ifx\SetFigFont\undefined%
\gdef\SetFigFont#1#2#3#4#5{%
  \reset@font\fontsize{#1}{#2pt}%
  \fontfamily{#3}\fontseries{#4}\fontshape{#5}%
  \selectfont}%
\fi\endgroup%
\begin{picture}(2763,1935)(2993,-587)
\put(3240,416){\makebox(0,0)[lb]{\smash{{\SetFigFont{8}{9.6}{\rmdefault}{\mddefault}{\updefault}{\color[rgb]{0,0,0}$z_0$}%
}}}}
\put(5202,1242){\makebox(0,0)[lb]{\smash{{\SetFigFont{8}{9.6}{\rmdefault}{\mddefault}{\updefault}{\color[rgb]{0,0,0}$z_0$}%
}}}}
\put(3221,-202){\makebox(0,0)[lb]{\smash{{\SetFigFont{8}{9.6}{\rmdefault}{\mddefault}{\updefault}{\color[rgb]{0,0,0}$\hat{C}_0$}%
}}}}
\put(5026,-215){\makebox(0,0)[lb]{\smash{{\SetFigFont{8}{9.6}{\rmdefault}{\mddefault}{\updefault}{\color[rgb]{1,1,1}$\hat{C}_0$}%
}}}}
\end{picture}%
\caption{Linear equivalence in $\ol{\M}_{1,1}(C)$} 
\label{config} 
\end{figure}

There is a similar stratification in the moduli space of scaled affine
curves which gives a homomorphism property for the quantum Kirwan map.
The stratification by combinatorial type is
$$ \ol{\M}_{n,1}(\bA) = \bigcup_\Gamma {\M}_{n,1,\Gamma}(\bA) $$
and the strata correspond to parenthesized expressions in symbols
$\kappa$ and $a_1,\ldots, a_n$ \cite{qk1}.  There is a canonical
identification 
$$\ol{\M}_{2,1}(\bA) \cong \P^1, \quad [z_1,z_2,\delta] \mapsto 
\delta(z_2 - z_1) $$
that is, the difference $z_2 - z_1$ in the affine coordinate defined
by the scaling $\delta$.  The two singular points in
$\ol{\M}_{2,1}(\bA)$ corresponding to the bubble trees in Figure
\ref{strata} give linearly equivalent divisors corresponding to an
expression
$\kappa(a_1 a_2) \sim \kappa(a_1) \kappa(a_2) .$
The divisor class relation in $\ol{\M}_{2,1}(\bA)$ pulls back to a
divisor class relation in any $\ol{\M}_{n,1}(\bA)$ implying the
homomorphism property of the linearized quantum Kirwan map \cite{qk1}.
This ends the Remark.
\end{remark} 

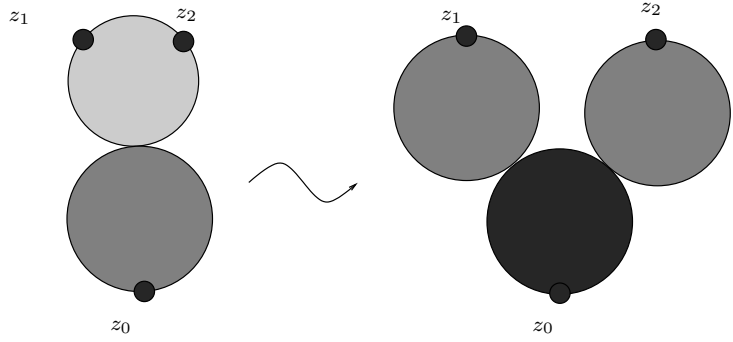
\begin{figure}[ht]
\begin{picture}(0,0)%
\includegraphics{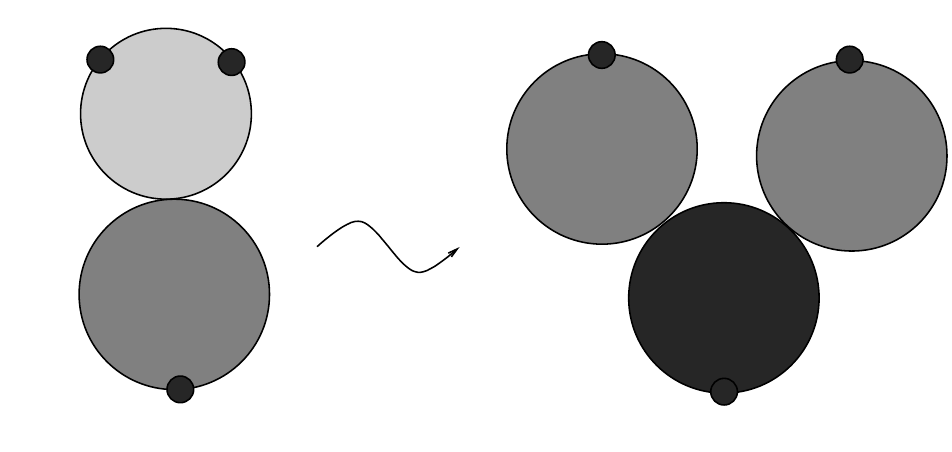}%
\end{picture}%
\setlength{\unitlength}{3947sp}%
\begingroup\makeatletter\ifx\SetFigFont\undefined%
\gdef\SetFigFont#1#2#3#4#5{%
  \reset@font\fontsize{#1}{#2pt}%
  \fontfamily{#3}\fontseries{#4}\fontshape{#5}%
  \selectfont}%
\fi\endgroup%
\begin{picture}(4553,2166)(2455,-1138)
\put(5759,-1086){\makebox(0,0)[lb]{\smash{{\SetFigFont{8}{9.6}{\rmdefault}{\mddefault}{\updefault}{\color[rgb]{0,0,0}$z_0$}%
}}}}
\put(3107,-1080){\makebox(0,0)[lb]{\smash{{\SetFigFont{8}{9.6}{\rmdefault}{\mddefault}{\updefault}{\color[rgb]{0,0,0}$z_0$}%
}}}}
\put(5183,866){\makebox(0,0)[lb]{\smash{{\SetFigFont{8}{9.6}{\rmdefault}{\mddefault}{\updefault}{\color[rgb]{0,0,0}$z_1$}%
}}}}
\put(6437,914){\makebox(0,0)[lb]{\smash{{\SetFigFont{8}{9.6}{\rmdefault}{\mddefault}{\updefault}{\color[rgb]{0,0,0}$z_2$}%
}}}}
\put(2470,859){\makebox(0,0)[lb]{\smash{{\SetFigFont{8}{9.6}{\rmdefault}{\mddefault}{\updefault}{\color[rgb]{0,0,0}$z_1$}%
}}}}
\put(3519,880){\makebox(0,0)[lb]{\smash{{\SetFigFont{8}{9.6}{\rmdefault}{\mddefault}{\updefault}{\color[rgb]{0,0,0}$z_2$}%
}}}}
\end{picture}%
\caption{Strata in affine scaled curves with two markings}
\label{strata}
\end{figure}

\section{Mumford stability}

In this section we review the relationship between the stack-theoretic
quotient and Mumford's geometric invariant theory quotient
\cite{mu:ge}.  First we introduce various Lie-theoretic notation. Let
$G$ be a connected complex reductive group with Lie algebra $\g$.
If \(G\) is a torus denote by 
$$ \g_\Z = \{ D \phi(1) \in \g \ | \ \phi \in \Hom(\C^\times,G) \},
\quad \g_\Q = \g_\Z \otimes_\Z \Q$$
the {\em coweight lattice} of derivatives of one-parameter subgroups
resp. rational one-parameter subgroups.  Dually denote by
$$ \g_\Z^\dual = \{ D \chi \in \g^\dual \ | \ \chi \in
\Hom(G,\C^\times) \}, \quad \g_\Q^\dual = \g_\Z^\dual \otimes_\Z \Q$$
the {\em weight lattice} of derivatives of characters of $G$ and the set
of {\em rational weights}, respectively.  More generally 
if $G$ is non-abelian then
we still denote by $\g_\Q$ the set of derivatives of rational
one-parameter subgroups.

The targets of our maps are quotient stacks defined as follows.  Let
$X$ be a smooth projective $G$-variety.  Let $X/G$ denote the quotient
stack, that is, the category fibered in groupoids whose fiber over a
scheme $S$ has objects pairs $v = (P,u)$ consisting of a principal
$G$-bundle $P \to S$ and a section $u: S \to P \times_G X$; and whose
morphisms are given by diagrams
$$ \begin{diagram} \node{P_1} \arrow{s} \arrow{e,t}{\phi} \node{P_2}
  \arrow{s} \\ \node{S_1}  \arrow{e,b}{\psi} \node{S_2} \end{diagram}, \quad \phi(X) \circ u_1 = u_2 \circ \psi $$
where $\phi(X) : P_1(X) \to P_2(X)$ denotes the map of associated
fiber bundles \cite{dm:irr},
\href{http://stacks.math.columbia.edu/tag/04UV}{Tag 04UV}
\cite{dejong:stacks}.

Mumford's geometric invariant theory quotient \cite{mu:ge} is
traditionally defined as the projective variety associated to the
graded ring of invariant sections of a linearization of the action in
the previous paragraph.  Let $\ti{X} \to X$ be a linearization, that
is, ample $G$-line bundle.  Then
$$ X \qu G := \Proj \left( \oplus_{k \ge 0} H^0(\ti{X}^k)^G \right)
.$$
Mumford \cite{mu:ge} realizes this projective variety as the quotient
of a {\em semistable locus} by an equivalence relation.  The
semistable locus consists of points $x \in X$ such that some tensor
power $\ti{X}^k, k > 0 $ of $\ti{X}$ has an invariant section
non-vanishing at $x$, while the unstable locus is the complement of
the semistable locus:
$$ X^{\ss} = \{ x \in X \ | \ \exists k > 0, \sigma \in
H^0(\ti{X}^k)^G, \quad \sigma(x) \neq 0 \}, 
\quad  X^{\us} := X - X^{\ss} .$$
A point $x \in X$ is {\em polystable} if its orbit is closed in the
semistable locus $\ol{Gx \cap X^{\ss}} = Gx \cap X^{\ss}$.  A point $x
\in X$ is {\em stable} if it is polystable and the stabilizer $G_x$ of
$x$ is finite.  In Mumford's definition the git quotient is the
quotient of the semistable locus by the {\em orbit equivalence
  relation}
$$ (x_1 \sim x_2) \iff \ol{Gx_1}\cap \ol{Gx_2} \cap X^{\ss} \neq
\emptyset. $$
Each semistable point is then orbit-equivalent to a unique polystable
point.  However, here we define the git quotient as the
stack-theoretic quotient
$$ X \qu G := X^{\ss}/G.$$ 
We shall always assume that $X^{\ss}/G$ is a Deligne-Mumford stack
(that is, the stabilizers $G_x$ are finite) in which case the coarse
moduli space of $X^{\ss}/G$ is the git quotient in Mumford's sense.
The Luna slice theorem \cite{luna:slice} implies that $X^{\ss}/G$ is
\'etale-locally the quotient of a smooth variety by a finite group,
and so has finite diagonal. By the Keel-Mori theorem \cite{km:quot},
explicitly stated in \cite[Theorem 1.1]{conrad:kl}, the morphism from
$X^{\ss}/G$ to its coarse moduli space is proper.  Since the coarse
moduli space of $X^{\ss}/G$ is projective by Mumford's construction,
it is proper, hence $X^{\ss}/G$ is proper as well.

Later we will need the following observation about the unstable locus.
As the quotient $X \qu G$ is non-empty, there exists an ample divisor
$D$ containing the unstable locus: take $D$ to be the vanishing locus
of any non-zero invariant section of $\ti{X}^k$ for some $k > 0$:
\begin{equation} \label{invsection} D = \sigma^{-1}(0), \quad \sigma \in H^0(\ti{X}^k)^G - \{ 0 \} .\end{equation} 

The Hilbert-Mumford numerical criterion \cite[Chapter 2]{mu:ge}
provides a computational tool to determine the semistable locus: A
point $x \in X$ is $G$-semistable if and only if it is
$\C^\times$-semistable for all one-parameter subgroups $\C^\times \to
G$.  Given an element $\lambda\in \g_\Z$ denote the
corresponding one-parameter subgroup $\C^\times \to G, \ z \mapsto
z^\lambda$.  Denote by
$$ x_\lambda := \lim_{z \to 0}  z^\lambda x $$
the limit under the one-parameter subgroup.  Let $\mu(x,\lambda) \in
\Z$ be the weight of the linearization $\ti{X}$ at $x_\lambda$ defined
by
$$ z \ti{x} = z^{\mu(x,\lambda)} \ti{x}, \quad \forall z \in
\C^\times, \ti{x} \in \ti{X}_{x_\lambda} .$$
By restricting to the case of a projective line one sees that the
point $x \in X$ is semistable if and only if $ \mu(x,\lambda) \leq 0$
for all $\lambda \in \g_\Z.$ Polystability is equivalent to
semistability and the additional condition $ \mu(x,\lambda) = 0 \iff
\mu(x,-\lambda) =0 .$ Stability is the condition that $ \mu(x,\lambda)
< 0$ for all $\lambda \in \g_\Z - \{ 0 \} .$

The Hilbert-Mumford numerical criterion can be applied
explicitly to actions on projective spaces as follows.
Suppose that $G$ is a torus and $X = \P(V)$ the projectivization of a
vector space $V$.  Let
$\ti{X} = \mO_X(1) \otimes \C_\theta $ 
be the $G$-equivariant line bundle given by tensoring the hyperplane
bundle $\mO_X(1)$ and the one-dimensional representation $\C_\theta$
corresponding to some weight $\theta \in \g_\Z^\dual$.
Recall if $p \in X$ is represented by a line $l \subset V$
then the fiber of $\mO_X(1) \otimes \C_\theta$ at $p$ is
$l^\dual\otimes \C_\theta$. 
In particular if $z^\lambda$ fixes $p$ then $z^\lambda$ scales $l$ by
some $z^{\mu(\lambda)}$ so that $z^\lambda \ti{x} =
z^{-\mu(\lambda)+\theta(\lambda)}\ti{x}$, for $\ti{x}\in l^\dual
\otimes \C_\theta$. Let $k = \dim(V)$ and decompose $V$ into weight
spaces $V_1,\ldots, V_k$ with weights $\mu_1,\dots,\mu_k\in
\g_\Z^\dual .$  Identify
$$H^2_G(X) \cong H^2_{\C^\times \times G}(V) \cong \Z \oplus
\g_\Z^\dual $$
Under this splitting the first Chern class $c_1^G(\ti{X})$ becomes
identified up to positive scalar multiple with the pair
 \begin{equation}\label{eqiv c_1}
 c_1^G(\ti{X}) \mapsto (1,\theta) \in \Z \oplus \g_\Z^\dual.
 \end{equation}
The following is essentially \cite[Proposition 2.3]{mu:ge}.

\begin{lemma} The semistable locus for the action of a torus $G$  on the projective space
$X =P(V)$ with weights $\mu_1,\ldots, \mu_k$ and linearization shifted
  by $\theta$ is $ X^{\ss} = \P(V)^{\ss} = \{ [x_1,\ldots,x_k] \in
  \P(V) \ | \ \on{hull} ( \{ \mu_i | x_i \neq 0 \}) \ni \theta \} .$ A
  point $x $ is polystable iff $\theta$ lies in the interior of the
  hull above, and stable if in addition the hull is of maximal
  dimension.
\end{lemma} 

\begin{proof}  The Hilbert-Mumford weights are computed as follows.  For
any non-zero $\lambda \in \g_\Z$, let
$$\nu(x,\lambda) := \min_i \left\{ - \mu_i(\lambda), x_i \neq 0 \right\} .$$
Then 
\begin{eqnarray*}
 z^\lambda [ x_1,\ldots, x_k ] &=& [ z^{\mu_1(\lambda)} x_1,\ldots, z^{\mu_k(\lambda)}
  x_k] \\
&=&  [ z^{\mu_1(\lambda) + \nu(x,\lambda)} x_1,\ldots,
  z^{\mu_k(\lambda) + \nu(x,\lambda)} x_k ] \end{eqnarray*}
and 
$$ (-\mu_i(\lambda) \neq \nu(x,\lambda)) \ \implies \ \left(\lim_{z \to
  0} z^{\mu_i(\lambda) + \nu(x,\lambda)} = 0 \right) .$$
Let
$$ x_\lambda := \lim_{z \to 0} z^{\lambda} x = \lim_{z \to 0}
[z^{\mu_i(\lambda)} x_i ]_{i =1}^k \in X $$
Then
$$ x_\lambda = [x_{\lambda,1},\ldots, x_{\lambda,k}],\quad
x_{\lambda,i} = \begin{cases} x_i & -\mu_i(\lambda) = \nu(x,\lambda)
  \\ 0 & \text{otherwise} \end{cases} .$$
The Hilbert-Mumford weight is therefore
\begin{equation} \label{hm} 
\mu(x,\lambda) = \nu(x,\lambda) + (\theta,\lambda) .\end{equation} 
By the Hilbert-Mumford criterion, the point $x$ is semistable if and
only if
$$ \nu(x,\lambda) := \min \{ - \mu_i(\lambda) \ | \ x_i \neq 0\} \leq
(- \theta,\lambda), \quad \forall \lambda \in \g_\Z - \{ 0 \} .$$
That is,
$$  (x \in X^{\ss}) \iff ( \theta \in \on{hull} 
\{ \mu_i \ |  \  x_i \neq 0\} ) .$$
This proves the claim about the semistable locus.  To prove the claim
about polystability, note that $\mu(x,\lambda) = 0 = \mu(x,-\lambda)$
implies that the minimum $\nu(x,\lambda)$ is also the maximum.  Thus
the only affine linear functions $\xi: \g^\dual \to \R$ which vanish
at $\theta$ are those $\xi$ that are constant on the hull of $\mu_i$
with $x_i$ nonzero.  This implies that the span of $\mu_i$ with $x_i$
non-zero contains $\theta$ in its relative interior.  The stabilizer
$G_x$ of $x$ has Lie algebra $\g_x$ the annihilator of the span of the
hull of the $\mu_i$ with $x_i \neq 0$.  So the stabilizer $G_x$ is
finite if and only if the span of $\mu_i$ with $x_i \neq 0$ is of
maximal dimension $\dim(G)$. This implies the claim on stability.
\end{proof}

We introduce the following notation for weight and coweight lattices.
As above $G$ is a connected complex reductive group with maximal torus $T$ and
$\g,\t$ are the Lie algebras of $G,T$ respectively.  Fix an invariant
inner product $( \ , \ ):\g \times \g \to \C$ on $\g$ inducing an
identification $\g \to \g^\dual$. By taking a multiple of the basic
inner product on each factor we may assume that the inner product
induces an identification $\t_\Q \to \t_\Q^\dual$.  Denote by
$$\Vert \cdot \Vert: \q_\Q \to \R_{\ge 0}, \quad \Vert \xi \Vert : =
(\xi , \xi)^{1/2}$$
the norm with respect to the induced metric.

Next recall the theory of Levi decompositions of parabolic subgroups
from Borel \cite[Section 11]{bo:lag}.  A parabolic subgroup $Q$ of $G$
is one for which $G/Q$ is complete, or equivalently, containing a
maximal solvable subgroup $B \subset G$.  Any parabolic $Q$ admits a
Levi decomposition $ Q = L(Q) U(Q) $ where $L(Q)$ denote a maximal
reductive subgroup of $Q$ and $U(Q)$ is the maximal unipotent
subgroup.  Let $\l(Q),\u(Q)$ denote the Lie algebras of $L(Q), U(Q)$.
Let $\g = \t \oplus \bigoplus_{ \alpha \in R(G)} \g_\alpha$ denote the
root space decomposition of $\g$, where $R(G)$ is the set of roots.
The Lie algebras $\l(Q),\u(Q)$ decompose into root spaces as
$$ \q = \t \oplus \bigoplus_{\alpha \in R(Q)} \g_\alpha, \quad \l(Q) =
\t \oplus \bigoplus_{\alpha \in R(Q) \cap -R(Q)} \g_\alpha, \quad
\u(Q) = \q/\l(Q) $$
where $R(Q) \subset R(G)$ is the set of roots for $\l(Q)$.  Let
$\z(Q)$ denote the center of $\l(Q)$ and
$$\z_+(Q) = \{ \xi \in \z(Q) \ | \ \alpha(\xi) \ge 0, \ \forall \alpha \in R(Q) \} $$ 
the {\em positive chamber} on which the roots of $Q$ are non-negative.
The Levi decomposition induces a homomorphism
\begin{equation} \label{piq} \pi_Q: Q \to Q/U(Q) \cong L(Q) .\end{equation}
This homomorphism has the following alternative description as a
limit.  Let $\lambda \in \z_+(Q) \cap \g_\Q$ be a positive coweight
and
$$ \phi_\lambda: \C^\times \to L(Q), \quad z \mapsto z^\lambda $$
the corresponding central one-parameter subgroup.  Then
$$ \pi_Q(g) = \lim_{z \to 0} \Ad(z^\lambda) g
.$$
In the case of the general linear group in which the parabolic
consists of block-upper-triangular matrices, this limit projects out
the off-block-diagonal terms.

The unstable locus admits a stratification by maximally destabilizing
subgroups, as in Hesselink \cite{hess:strat}, Kirwan \cite{ki:coh},
and Ness \cite{ne:st}.  The stratification reads
\begin{equation} \label{kn} X = \bigcup_{\lambda \in \cC(X)} X_\lambda, \quad X_\lambda = G
\times_{Q_\lambda} Y_\lambda, \quad Y_\lambda \mapsto Z_\lambda
\ \text{affine fibers} \end{equation}
where $Y_\lambda,Z_\lambda,Q_\lambda,\cC(X)$ are defined as
follows. For each fixed point component $Z_\lambda$ of $z^{\lambda}$
there exist a weight $\mu(\lambda)$ so $z^{\lambda}$ acts on $\ti{X} |
Z_\lambda$ with weight $\mu(\lambda)$:
$$ z^\lambda \ti{x} = z^{\mu(\lambda)} \ti{x}, \quad \forall \ti{x}
\in \ti{X} | Z_\lambda .$$
Define
\begin{equation} \label{CX} \cC(X) = \{ \lambda \in \t_+ \ | \ \exists Z_\lambda,
\ \mu(\lambda) = (\lambda, \lambda) \} \end{equation}
using the metric, where $\t_+$ is the closed positive Weyl chamber.
The variety $Y_\lambda$ is the set of points that flow to $Z_\lambda$
under $z^{\lambda}, z \to 0$:
$$ Y_\lambda = \left\{  x \in X  \ | \ \lim_{z \to 0} z^\lambda x \in Z_\lambda \right\} $$
The group $Q_\lambda$ is the parabolic of group elements that have a
limit under $\Ad(z^{\lambda})$ as $z \to 0$:
$$ Q_\lambda = \left\{ g \in G \ | \ \exists \lim_{z \to 0}
\Ad(z^\lambda) g \in G \right\} .$$
Then $Y_\lambda$ is a $Q_\lambda$-variety; and $X_\lambda$ is the
flow-out of $Y_\lambda$ under $G$.  By taking quotients we obtain a
stratification of the quotient stack by locally-closed substacks
$$ X / G = \bigcup_{\lambda \in \cC(X)} X_\lambda / G .$$
This stratification was used in Teleman \cite{te:qu} to give a formula
for the sheaf cohomology of bundles on the quotient stack.

\section{Kontsevich stability}

In this section we recall the definition of Kontsevich's moduli stacks
of stable maps \cite{ko:lo} as generalized to orbifold quotients by
Chen-Ruan \cite{cr:orb} and in the algebraic setting by
Abramovich-Graber-Vistoli \cite{agv:gw}.  Let $X$ be a smooth
projective variety.  Recall that a {\em prestable map} with target $X$
consists of a prestable curve $C \to S$, a morphism $u: C \to X$, and
a collection $z_1,\ldots,z_n : S \to C$ of distinct non-singular
points called {\em markings}.  An automorphism of a prestable map
$(C,u,\ul{z})$ is an automorphism
$$\varphi:C \to C, \quad \varphi \circ u = u, \quad \varphi(z_i) = z_i, \quad i = 1,\ldots, n .$$
A prestable map $(C,u,\ul{z})$ is {\em stable} if the number $\#
\Aut(C,u,\ul{z})$ of automorphisms is finite.  For $d \in H_2(X,\Z)$
we denote by $\ol{\M}_{g,n}(X,d)$ \label{mgn} the moduli stack of
stable maps $(C,u,\ul{z})$ of genus $g = \on{genus}(C)$ and class $d =
v_*[C]$ with $n$ markings.

The notion of stable map generalizes to orbifolds \cite{cr:orb},
\cite{agv:gw} as follows.  These definitions are needed for the
construction of the moduli stack of affine gauged maps in the case
that the git quotient is an orbifold, but not if the quotient is free.
First we recall the notion of twisted curve:

\begin{definition} \label{twistedcurve} {\rm (Twisted curves)} 
Let $S$ be a scheme.  An {\em $n$-marked twisted curve} over $S$ is a
collection of data $(f: \cC \to S, \{ {\mathcal z}_i \subset \cC
\}_{i=1}^n)$ such that
\begin{enumerate} 
\item {\rm (Coarse moduli space)} $\cC$ is a proper stack over $S$
  whose geometric fibers are connected of dimension $1$, and such that
  the coarse moduli space of $\cC$ is a nodal curve $C$ over $S$.
\item {\rm (Markings)} The ${\mathcal z}_i \subset \cC$ are closed
  substacks that are gerbes over $S$, and whose images in $C$ are
  contained in the smooth locus of the morphism $C \to S$.
\item {\rm (Automorphisms only at markings and nodes)} If $\cC^{ns}
  \subset \cC$ denotes the {\em non-special locus} given as the
  complement of the ${\mathcal z}_i$ and the singular locus of $\cC
  \to S$, then $\cC^{ns} \to C$ is an open immersion.
\item {\rm (Local form at smooth points)} If $p \to C$ is a geometric
  point mapping to a smooth point of $C$, then there exists an integer
  $r$, equal to $1$ unless $p$ is in the image of some ${\mathcal
    z}_i$, an \'etale neighborhood $\Spec(R) \to C$ of $p$ and an
  \'etale morphism $\Spec(R) \to \Spec_S(\mO_S[x])$ such that the
  pull-back $\cC \times_C \Spec(R)$ is isomorphic to $ \Spec(R[z]/z^r
  = x )/\mu_r .$
\item {\rm (Local form at nodal points)} If $p \to C$ is a geometric
  point mapping to a node of $C$, then there exists an integer $r$, an
  \'etale neighborhood $\Spec(R) \to C$ of $p$ and an \'etale morphism
  $\Spec(R) \to \Spec_S(\mO_S[x,y]/(xy - t))$ for some $t \in \mO_S$
  such that the pull-back $\cC \times_C \Spec(R)$ is isomorphic to $
  \Spec(R[z,w]/zw = t', z^r = x, w^r = y )/\mu_r $ for some $t' \in
  \mO_S$.
\end{enumerate}
\end{definition} 

Next we recall the notion of twisted stable maps.  Let $\XX$ be a
proper Deligne-Mumford stack with projective coarse moduli space $X$.
Algebraic definitions of twisted curve and twisted stable map to a
$\XX$ are given in Abramovich-Graber-Vistoli \cite{agv:gw},
Abramovich-Olsson-Vistoli \cite{aov:twisted}, and Olsson
\cite{ol:logtwist}.

\begin{definition} 
A {\em twisted stable map} from an $n$-marked twisted curve $(\pi :
\cC \to S, ( {\mathcal z}_i \subset \cC )_{i=1}^n )$ over $S$ to $\XX$
is a representable morphism of $S$-stacks $ u: \cC \to \XX $ such that
the induced morphism on coarse moduli spaces $ u_c: C \to X $ is a
stable map in the sense of Kontsevich from the $n$-pointed curve $(C,
\ul{z} = (z_1,\ldots, z_n ))$ to $X$, where $z_i$ is the image of
${\mathcal z}_i$.  The {\em homology class} of a twisted stable curve
is the homology class $u_* [ \cC_s] \in H_2(X,\Q)$ of any fiber
$\cC_s$.
\end{definition} 

\noindent Twisted stable maps naturally form a $2$-category.  Every $2$-morphism
is unique and invertible if it exists, and so this $2$-category is
naturally equivalent to a $1$-category which forms a stack over
schemes \cite{agv:gw}.

\begin{theorem} (\cite[4.2]{agv:gw}) The stack $\ol{\M}_{g,n}(\XX)$ of
  twisted stable maps from $n$-pointed genus $g$ curves into $\XX$ is
  a Deligne-Mumford stack.  If $\XX$ is proper, then for any $c > 0$
  the union of substacks $\ol{\M}_{g,n}(\XX,d)$ with homology class $d
  \in H_2(\XX,\Q)$ satisfying $(d, c_1(\ti{X}))< c$ is proper.
\end{theorem} 

The Gromov-Witten invariants takes values in the cohomology of the
{\em inertia stack}
$$ \cI_\XX := \XX \times_{\XX \times \XX} \XX $$
where both maps are the diagonal.  The objects of $\cI_\XX$ may be
identified with pairs $(x,g)$ where $x \in \XX$ and $g \in
\Aut_\XX(x)$.  For example, if $\XX = X/G$ is a global quotient by a
finite group then
$$ \cI_\XX = \bigcup_{[g] \in G/\Ad(G)} X^g/Z_g $$
where $G/\Ad(G)$ denotes the set of conjugacy classes in $X$ and $Z_g$
is the centralizer of $g$.  Let $\mu_r = \Z/r\Z$ denote the group of
$r$-th roots of unity.  The inertia stack may also be written as a hom
stack \cite[Section 3]{agv:gw}
$$ \cI_\XX = \cup_{r > 0} \cI_{\XX,r}, \quad \cI_{\XX,r} :=
\Hom^{\on{rep}}(B\mu_r, \XX) . $$
The classifying stack $B\mu_r$ is a Deligne-Mumford stack and if $\XX$
is a Deligne-Mumford stack then
$$ \ol{\cI}_\XX := \cup_{r > 0} \ol{\cI}_{\XX,r}, \quad
\ol{\cI}_{\XX,r} := \cI_{\XX/r}/ B\mu_r . $$
is the {\em rigidified inertia stack} of representable morphisms from
$B \mu_r$ to $\XX$, see \cite[Section 3]{agv:gw}.  There is a
canonical quotient cover $\pi: \cI_\XX \to \ol{\cI}_\XX$ which is
$r$-fold over $\ol{\cI}_{\XX,r}$.  Since the covering is finite, from
the point of rational cohomology there is no difference between
$\ol{\cI}_{\XX}$ and $\cI_{\XX}$; that is, pullback induces an
isomorphism
$$ \pi^* : H^*(\ol{\cI}_\XX,\Q) \to H^*(\cI_\XX,\Q) .$$
For the purposes of defining orbifold Gromov-Witten invariants,
$\ol{\cI}_\XX$ can be replaced by $\cI_\XX$ at the cost of additional
factors of $r$ on the $r$-twisted sectors. If $\XX = X/G$ is a global
quotient of a scheme $X$ by a finite group $G$ then
$$ \ol{\cI}_{X/G} = \coprod_{(g)} X^{g}/ (Z_g/ \lan g \ran) $$
where $\lan g \ran \subset Z_g$ is the cyclic subgroup generated by
$g$.
For example, suppose that $X$ is a polarized linearized projective
$G$-variety such that $X \qu G$ is locally free.  Then
$$ \cI_{X \qu G} = \coprod_{(g)} X^{\ss,g}/ Z_g $$
where $X^{\ss,g}$ is the fixed point set of $g \in G$ on $X^{\ss}$,
$Z_g$ is its centralizer, and the union is over all conjugacy classes,
$$ \ol{\cI}_{X \qu G} = 
\coprod_{(g)} X^{\ss,g}/ (Z_g/ \lan g \ran)  $$
where $\lan g \ran$ is the (finite) group generated by $g$.  The
moduli stack of twisted stable maps admits evaluation maps to the
rigidified inertia stack
$$ \ev: \ol{\M}_{g,n}(\XX) \to \ol{\cI}_\XX^n, 
\quad  \ol{\ev}: \ol{\M}_{g,n}(\XX) \to \ol{\cI}_\XX^n, 
 $$
where the second is obtained by composing with the involution
$\ol{\cI}_{\XX} \to \ol{\cI}_{\XX} $ induced by the map $\mu_r \to
\mu_r, \zeta \mapsto \zeta^{-1}$.  

Constructions of Behrend-Fantechi \cite{bf:in} provide the stack of stable maps
with virtual fundamental classes.  The virtual fundamental classes
$$[\ol{\M}_{g,n,\Gamma}(\XX,d)] \in H(\ol{\M}_{g,n}(\XX),\Q) $$
(where the right-hand-side denotes the singular homology of the coarse
  moduli space) satisfy the splitting axioms for morphisms of modular
  graphs similar to those in the case that $X$ is a variety.  Orbifold
  Gromov-Witten invariants are defined by virtual integration of
  pull-back classes using the evaluation maps above.  The orbifold
  Gromov-Witten invariants satisfy properties similar to those for
  usual Gromov-Witten invariants, after replacing rescaling the inner
  product on the cohomology of the inertia stack by the order of the
  stabilizer.  The definition of orbifold Gromov-Witten invariants
  leads to the definition of orbifold quantum cohomology as follows.

  \begin{definition}\label{d:oqh} {\rm (Orbifold quantum cohomology)}
  To each component $\XX_k$ of ${\cI}_\XX$ is assigned a rational
  number $\age(\XX_k)$ as follows.  Let $(x,g)$ be an object in
  $\XX_k$.  The element $g$ acts on $T_x \XX$ with eigenvalues
  $\alpha_1, \ldots,\alpha_n$ with $ n = \dim(\XX)$.  Let $r$ be the
  order of $g$ and define $s_j \in \{ 0,\ldots, r - 1 \}$ by $\alpha_j
  = \exp( 2\pi i s_j / r)$.  The {\em age} is defined by $ \age(\XX_k)
  = (1/r) \sum_{j=1}^n s_j .$ Let 
$$ \Lambda_\XX = \left\{ \sum b_i q^{d_i} \ | \ b_i \in \Q, d_i \in
  H_2(\XX,\Q), \  
  \forall c > 0, \# \{ d_i \ | \ (d_i,c_1(\ti{X})) < c \} < \infty
  \right\} $$
  denote the Novikov field of sums of formal symbols $q^{d_i}, d_i \in
  H_2(\XX,\Q)$ where for each $c> 0$, only finitely many $q^{d_i}$
  with $(d_i,c_1(\ti{X})) < c$ have non-zero coefficient.  Denote the
  quantum cohomology
$$ QH(\XX) = \bigoplus QH^\bullet(\XX), \quad QH^\bullet(\XX) =
\bigoplus_{\XX_k \subset \cI_\XX} H^{\bullet + 2 \age(\XX_k)}(\XX_k)
\otimes \Lambda_\XX .$$
\end{definition}

\noindent The genus zero Gromov-Witten invariants define on $QH(\XX)$
the structure of a Frobenius manifold \cite{cr:orb}, \cite{agv:gw}.

\section{Mundet stability} 

In this section we explain the Ramanathan condition for semistability
of principal bundles \cite{ra:th} and its generalization to maps to
quotients stacks by Mundet \cite{mund:corr}, and the quot-scheme and
stable-map compactification of the moduli stacks.  

\subsection{Ramanathan stability} 

Morphisms from a curve to a quotient of a point by a reductive group
are by definition principal bundles over the curve.  Bundles have a
natural semistability condition introduced half a century ago by
Mumford, Narasimhan-Seshadri, Ramanathan and others in terms of {\em
  parabolic reductions} \cite{ra:th}.  First we explain stability for
vector bundles.  A vector bundle $E \to C$ of degree zero over a
smooth projective curve $C$ is semistable if there are no sub-bundles
of positive degree:
$$ (E \ \text{semistable} ) \quad \text{iff} \quad ( \deg(F) \leq 0, \quad \forall
F \subset E \ \text{sub-bundles}) .$$
A generalization of the notion of semistability to principal bundles
is given by Ramanathan \cite{ra:th} in terms of {\em parabolic
  reductions}.  A parabolic reduction of $P$ consists of a pair
$$Q \subset G, \quad \sigma : C \to P/Q$$
of a parabolic subgroup of $G$, that is and a section $\sigma: C \to
P/Q$.  Denote by $\sigma^* P \subset P$ the pull-back of the
$Q$-bundle $P \to P/Q$, that is, the reduction of structure group of
$P$ to $Q$ corresponding to $\sigma$.  Associated to the homomorphism
$\pi_Q$ of \eqref{piq} is an {\em associated graded} bundle $\Gr(P) :=
\sigma^*P \times_Q L(Q) \to C$ with structure group $L(Q)$.  In the
case that $P$ is the frame bundle of a vector bundle $E \to C$ of rank
$r$, that is,
$$ P = \cup_z P_z, \quad P_z = \{ (e_1,\ldots, e_r) \in E_z^r \ |
\ e_1 \wedge \ldots \wedge e_r \neq 0 \} $$
a parabolic reduction of $P$ is equivalent to a flag of
sub-vector-bundles of $E$
$$ \{0 \} \subset E_{i_1} \subset E_{i_2} \subset \ldots \subset
E_{i_l} \subset E. $$
Explicitly the parabolic reduction $\sigma^* P$ given by frames
adapted to the flag:
$$ \sigma(z) = \{ (e_1,\ldots,e_r) \in E_z^r \ | \ e_j \in E_{i_k,z},
\ \forall j \leq i_k, k = 1,\ldots, l \} .$$
Conversely, given a parabolic reduction the associated vector bundle
has a canonical filtration.

An analog of the degree of a sub-bundle for parabolic reductions is
the degree of a line bundle defined as follows.  Given $\lambda \in
\g_\Z - \{ 0 \}$ we obtain from the identification $\g \to \g^\dual$ a
rational weight $\lambda^\vee$.  Denote the corresponding characters
$ \chi_\lambda: L(Q) \to \C^\times$ and $ \chi_\lambda\circ \pi_Q: Q
\to \C^\times .$
Consider the associated line bundle over $C$ defined by 
$P(\C_{\lambda^\vee}) := \sigma^* P \times_Q \C_{\lambda^\vee} .$
The {\em Ramanathan weight} \cite{ra:th} is the degree of the line
bundle $P(\C_{\lambda^\vee}) $, that is,
%
$$ \mu_{BG}(\sigma,\lambda) := \lan [C],
c_1(P(\C_{\lambda^\vee}))\ran \in \Z
.$$
The bundle $P \to C$ is {\em Ramanathan semistable} if for all 
$(\sigma,\lambda)$ with $\lambda $ dominant, 
$$ \mu_{BG}(\sigma,\lambda) \leq 0 , \quad \forall (\sigma,\lambda) .$$
As in the case of vector bundles, it suffices to check semistability
for all reductions to {\em maximal parabolic} subgroups.  In fact, any
dominant weight may be used in the definition of
$\mu_{BG}(\sigma,\lambda)$, which shows that Ramanathan semistability
is independent of the choice of invariant inner product on the Lie
algebra and one obtains the definition given in Ramanathan
\cite{ra:th}.

\subsection{Mundet semistability} 

The Mundet semistability condition generalizes Ramanathan's condition
to morphisms from a curve to the quotient stack \cite{mund:corr},
\cite{schmitt:univ}.  Let
$$(p: P \to C, u: C \to P(X)) \in \on{Obj}(\Hom(C,X/G)) $$
be a gauged map.  Let $(\sigma,\lambda)$ consist of a parabolic
reduction $\sigma: C \to P/Q$ and a positive coweight $\lambda \in
\z_+(Q)$.  Consider the family of bundles $ P^\lambda \to S :=
\C^\times $ obtained by conjugating by $z^\lambda$.  That is, if $P$
is given as a cocycle in nonabelian cohomology with respect to a
covering $\{ U_i \to X \}$
$$ [P] = [\psi_{ij} :
(U_i \cap U_j) \to G] \ \in \ H^1(C,G) $$
then the twisted bundle is given by 
$$ [P^\lambda] = [ z^{\lambda} \psi_{ij}
z^{-\lambda}: (U_i \cap U_j) \to G ] \ \in \  H^1(C \times S,G) .$$
Define a family of sections
$$ u^\lambda: S \times C \to P^\lambda(X) $$
by multiplying $u$ by $z^\lambda, z \in \C^\times$.  This family has
an extension over $s = \infty$ called the {\em associated graded}
bundle and stable section
\begin{equation} \label{assocgrad} \Gr(P) \to C, \quad \Gr(u): \hat{C} \to \Gr(P)(X) \end{equation} 
whose bundle $\Gr(P)$ agrees with the definition of associated graded
above.  Note that the associated graded section $\Gr(u)$ exists by
properness of the moduli space of stable maps to $\Gr(P)(X)$.  The
composition of $\Gr(u)$ with projection $\Gr(P)(X)\to C$ is a map of
degree one; hence there is a unique component $\hat{C}_0$ of $\hat{C}$
that maps isomorphically onto $C$.  The construction above is
$\C^\times$-equivariant and in particular over the central fiber $z =
0$ the group element $z^\lambda$ acts by an automorphism of $\Gr(P)$
fixing $\Gr(u)$ up to automorphism of the domain.

For each pair of a parabolic reduction and one-parameter subgroup as
above, the Mundet weight is a sum of {\em Ramanathan} and {\em
  Hilbert-Mumford} weights.  To define the Mundet weight, consider the
action of the automorphism $z^\lambda$ on the associated graded
$\Gr(P)$.  The automorphism of $X$ by $z^\lambda$ is $L(Q)$-invariant
and so defines an automorphism of the associated line bundle $\Gr(u)^*
P(\ti{X}) \to \Gr(C)$.  The weight of the action of $z^\lambda$ on the
fiber of $\Gr(u)^* P(\ti{X})$ over the root component $\hat{C}_0$
is the {\em Hilbert-Mumford weight}
$$ \mu_X (\sigma,\lambda) \in \Z, \quad z^\lambda \ti{x} =
z^{\mu_X(\sigma,\lambda)} \ti{x}, \quad \forall \ti{x} \in
(\Gr(u)|_{\hat{C}_0})^* \Gr(P) \times_G \ti{X} .$$

\begin{definition}  {\rm (Mundet stability)}  
Let $(P,u)$ be a gauged map from a smooth projective curve $C$ to the
quotient stack $X/G$.  The {\em Mundet weight} of a parabolic
reduction $\sigma$ and dominant coweight $\lambda$ is
$$ \mu(\sigma,\lambda) = \mu_{BG}(\sigma,\lambda) +
\mu_X(\sigma,\lambda) \in \Z .$$
The gauged map $(P,u)$ is Mundet {\em semistable} resp. {\em stable}
if and only if
$$ \mu(\sigma,\lambda) \leq 0, \ \text {resp.} \ < 0, \quad \forall
    (\sigma,\lambda) .$$
A pair $(\sigma,\lambda)$ such that $ \mu(\sigma,\lambda ) \ge 0 $ is
a {\em destabilizing pair}.  A pair $(P,u)$ is {\em polystable} iff
\begin{equation} \label{polystable}  \mu(\sigma,\lambda) = 0 \iff \mu(\sigma,-\lambda) = 0, \quad \forall
(\sigma,\lambda) .\end{equation}
That is, a pair $(P,u)$ is polystable if for any destabilizing pair
the opposite pair is also destabilizing. 
\end{definition}

More conceptually the semistability condition above is the
Hilbert-Mumford stability condition adapted to one-parameter subgroups
of the complexified gauge group, as explained in \cite{mund:corr}.
Semistability is independent of the choice of invariant inner product
as follows for example from the presentation of the semistable locus
in Schmitt \cite[Section 2.3]{schmitt:git}.

We introduce notation for various moduli stacks.  Let $\M^G(C,X)$
denote the moduli space of Mundet semistable pairs; in general,
$\M^G(C,X)$ is an Artin stack as follows from the git construction
given in Schmitt \cite{schmitt:univ,schmitt:git} or the more general
construction of hom stacks in Lieblich \cite[2.3.4]{lieblich:rem}.
For any $d \in H_2^G(X,\Z)$, denote by $\M^G(C,X,d)$ the moduli stack
of pairs $v = (P,u)$ with
$$v_* [C] := (\phi \times_G \on{id}_X)_* u_* [C] = d \in
H_2^G(X,\Z) $$
where $\phi:P \to EG$ is the classifying map.  

\subsection{Compactification} 

Schmitt \cite{schmitt:univ} constructs a Grothendieck-style
compactification \label{quots} of the moduli space of
Mundet-semistable obtained as follows.  Suppose $X$ is projectively
embedded in a projectivization of a representation $V$, that is $ X
\subset \P(V)$.  Any section $u: C \to P(X)$ gives rise to a line
sub-bundle
$ L := u^* (\mO_{\P(V)} (-1) \to \P(V))$
of the associated vector bundle $P \times_G V$.  From the inclusion
$\iota:L \to P(V)$ we obtain by dualizing a surjective map
$$ j: P(V^\dual) := P \times_G V^\dual \to L^\dual .$$
A {\em bundle with generalized map} in the sense of Schmitt
\cite{schmitt:git} is a pair $(P,j)$ such that $j$ has base points in
the sense that 
$$\zeta \in C \ \text{basepoint} \ \iff ( (\rank(j_\zeta):
P(V)_\zeta^\dual \to L_\zeta^\dual) = 0) .$$
Schmitt \cite{schmitt:git} shows that the Mundet semistability
condition extends naturally to the moduli stack of bundles with
generalized map.  Furthermore, the compactified moduli space
$\ol{\M}^{\quot,G}(C,X)$ is projective, in particular
proper.

Schmitt's construction of the moduli space of bundles with generalized
maps uses geometric invariant theory.  After twisting by a
sufficiently positive bundle we may assume that $P(V^\dual)$ is
generated by global sections.  A collection of sections $s_1,\ldots,
s_l$ generating $P(V^\dual)$ is called a {\em level structure}.
Equivalently, an \(l\)-level structure is a surjective morphism $ q:
\mO_C^{\oplus l} \to P(V^\dual) .$ Denote by
$\M^{G,\lev}(C,\P(V))$ \label{level} the stack of gauged maps to
$\P(V)$ with \(l\)-level structure.  The group $GL(l)$ acts on the stack of
\(l\)-level structures, with quotient
\begin{equation} \label{schmittgit} \M^{G,\lev}(C,\P(V)) / GL(l) = \M^G(C,\P(V)) .\end{equation}
Denote by $\M^{G,\lev}(C,X) \subset \M^{G,\lev}(C,\P(V))$ the substack
whose sections $u: C \to \P(V)$ have image in $P(X) \subset P(\P(V))$.
Then by restriction we obtain a quotient presentation
$$ \M^{G,\lev}(C,X) / GL(l) = \M^G(C,X) .$$
Allowing the associated quotient $P \times_G V^\dual \to P \times_G
L^\dual$ to develop base points gives a compactified moduli stack of
gauged maps with level structure $\ol{\M}^{G,\quot,\lev}(C,X)$.
Schmitt \cite{schmitt:univ,schmitt:git} shows that the stack
$\ol{\M}^{G,\quot,\lev}(C,X)$ has a canonical linearization and the
git quotient $\ol{\M}^{G,\quot,\lev}(C,X) \qu GL(l)$ defines a
compactification $\ol{\M}^{G,\quot}(C,X)$ of $\M^G(C,X)$ independent
of the choice of $l$ as long as $l$ is sufficiently large.  A version
of the quot-scheme compactification with markings is obtained by
adding tuples of points to the data.  That is,
$$ \ol{\M}^{G,\quot}_n(C,X) := \ol{\M}^{G,\quot}(C,X) \times
\ol{\M}_n(C) $$
where we recall that $\ol{\M}_n(C)$ is the moduli stack of stable maps
$p: \hat{C} \to C$ of class $[C]$ with $n$ markings and genus that of
$C$.  The orbit-equivalence relation can be described more
naturally in terms of {\em $S$-equivalence}: Given a family
$(P_S,u_S)$ of semistable gauged maps over a scheme $S$, such that the
generic fiber is isomorphic to some fixed $(P,u)$, then we declare
$(P,u)$ to be $S$-equivalent to $(P_s,u_s)$ for any $s \in S$.  Any
equivalence class of semistable gauged maps has a unique
representative that is polystable, by the git construction in Schmitt
\cite[Remark 2.3.5.18]{schmitt:univ}.  From the construction
evaluation at the markings defines maps to the quotient stack
$$ \ol{\M}_n^{G,\quot}(C,X,d) \to (V/\C^\times)^n, \quad ((p \circ
z_i)^*L, j \circ p \circ z_i) $$
rather than to the git quotient $X^n \subset \P(V)^n$.\footnote{The
  Ciocan-Fontanine-Kim-Maulik \cite{cf:st} moduli space of {\em stable
    quotients} remedies this defect by imposing a stability condition
  at the marked points $z_1,\ldots, z_n \in C$.  The moduli stack then
  admits a morphism to $\ol{\cI}_{X \qu G}^n$ by evaluation at the
  markings.}

\begin{example}\label{toric example} {\rm (Mundet semistable maps in the toric case)}
  If $G$ is a torus and $X = \P(V)$ then we can give an explicit
  description of Schmitt's quot-scheme compactification
  $\ol{\M}^{G,\quot}(C,X,d) $ of Mundet semistable maps
  \cite{schmitt:univ}.

Specifically let $X = \P(V)$ where $V$ is a $k$-dimensional vector space and
\begin{equation}  \label{decompose}  V = \bigoplus_{i=1}^k V_i \end{equation} 
is the decomposition of $V$ into weight spaces $V_i$ with weight
$\mu_i \in \g_\Z^\dual$.  

A point of $\M^{G}(C,X,d)$ is a pair $(P,u)$:
\[
P\to C \ \ \ \ \ \  u\colon C \to P(X),
\]
where $P$ is a $G$-bundle and $u$ is a section. We consider $u$ as a morphism $\widetilde{u} \colon L\to P(V)$ with $L\to C$ a line bundle \cite[Theorem 7.1]{har}. Via the decomposition of $V$, we can write $\widetilde{u}$ as a $k$-tuple:
\[
(\widetilde{u}_1, \dotsc, \widetilde{u}_k) \in \bigoplus_{i=1}^k H^0(P(V_i)\otimes L^\dual).
\]
The compactification $\ol{\M}^{G,\quot}(C,X,d)$ is obtained by allowing the $\widetilde{u}_i$ to have simultaneous zeros:
\[
\widetilde{u}_1^{-1}(0) \cap \dotsb \cap
\widetilde{u}_k^{-1}(0) \neq \emptyset.
\]
We make use of this example later on so we collect a few results about $\ol{\M}^{G,\quot}(C,X,d)$ below.
\end{example}
Recall \eqref{eqiv c_1} there is a projection $H^2_G(X)\to H^2(B G)=\g_\Z^\dual $ and similarly we have $H_2^G(X)\to H_2(B G) = \g_\Z$.  Associated to $v=(P,u)$ is the discrete data:
\begin{itemize}
\item[--] $v_*[C]=d \in H_2^G(X,\Z)$ and its image $d(P) \in H_2(BG)$
\item[--] $c_1^G(\widetilde{X}) \in H^2_G(X)$ and its image $\theta \in H^2(B G)$
\item[--] $d(u):= -c_1(L) \in H^2(C,\Z) \cong \Z$.
\end{itemize}
Note that $d(P)$ is the degree of $P$; that is, $d(P) =
c_1(P) \in \g_\Z$ under the identification
$H^2(C,\g_\Z)\cong \g_\Z$. We can now state the 
following:

\begin{lemma}\label{torus action on  P(V)} Let $G$ be a
  torus acting on a vector space $V$. Let $V = \bigoplus_{i=1}^kV_i$
  be its decomposition into weight spaces with weights $\mu_1, \dotsc,
  \mu_k$.  The Mundet semistable locus consists of pairs $(P,u)$ such
  that
\begin{equation} \label{oneeq} \on{hull} ( \{
		- d(P)^\dual + \mu_i | \ti{u}_i \neq 0 \}) \ni
		\theta. \end{equation} 
Furthermore let $W = \bigoplus_{i=1}^k H^0(P(V_i)\otimes L^\dual)$ and
let $W^{ss}$
consist of $(\widetilde{u}_1, \dotsc, \widetilde{u}_k)$ such that
\eqref{oneeq} holds. Then
 $ \ol{\M}^{G,\quot}(C,X,d) \cong W^{ss}/G. $
%
\end{lemma}

\begin{proof}
Since $G$ is abelian, $\Gr(P) = P$ for any
  pair $(\lambda,\sigma)$.  It follows that for any $\lambda \in
  \g_\Q$, the Mundet weight is
$$ \mu(\sigma,\lambda) := \min_i \{ (d(P),\lambda) -
  \mu_i(\lambda) + \theta(\lambda), \ti{u}_i \neq 0\} .$$
Hence the semistable locus is the space of pairs $(P,u)$
where
$$ \on{hull} ( \{ - d(P)^\dual + \mu_i | \ti{u}_i \neq 0 \})
\ni \theta .$$
This proves the first claim.
The second claim is an immediate consequence.  


\end{proof}

\begin{example}
Consider $G = \C^\times $ and $V =
\C^k$. Then
$$\deg(  P(V_i)\otimes L^\dual) =  \deg(P(V_i)) - \deg(L) = d(P) + d(u),
\quad i = 1,\ldots, k .$$
It follows that the moduli stack admits an isomorphism
$$ \ol{\M}^{G,\quot}(C,X,d) \cong \C^{k(d(P) + d(u) + 1),\times} /
\C^\times \cong \P^{k(d(P) + d(u) + 1) - 1}.$$
This moduli stack is substantially simpler in topology than
the moduli space of stable maps to $C \times X\qu G$, despite the dramatically more complicated stability condition. This ends the example.
\end{example}

A {\em Kontsevich-style compactification} of the stack of
Mundet-semistable gauged maps which admits evaluation maps as well as
a Behrend-Fantechi virtual fundamental class \cite{cross} is defined
as follows.  The objects in this compactification allow {\em stable
  sections}, that is, stable maps $u : \hat{C} \to P(X) $ whose
composition with $P(X) \to C$ has class $[C]$.  Thus objects of
$\ol{\M}^G_n(C,X)$ \label{mss} are triples $(P, \hat{C}, u,\ul{z})$
consisting of a $G$-bundle $P \to C$, a projective nodal curve
$(\hat{C},\ul{z})$, and a stable map $u: \hat{C} \to P \times_G X$
whose class projects to $[C] \in H_2(C,\Z)$.  Morphisms are the
obvious diagrams.  To see that this category forms an Artin stack,
note that the moduli stack of bundles $\Hom(C,BG)$ has a universal
bundle
$$U \to C \times \Hom(C,BG) .$$ 
Consider the associated $X$-bundle
$$U \times_G X \to C \times \Hom(C,BG)  .$$  
The stack $\ol{\M}_n^G(C,X)$ is a substack of the stack of stable maps to $U
\times_G X$, and is an Artin stack by e.g. Lieblich
\cite[2.3.4]{lieblich:rem}, see \cite{qk2} for more details.  Note
that hom-stacks are not in general algebraic \cite{bhatt}.

Properness of the Kontsevich-style compactification follows from a
combination of Schmitt's construction and the Givental map.  A proper
relative Givental map is described in Popa-Roth \cite{po:stable}, and
in this case gives a morphism
\begin{equation} \label{givmor} 
\ol{\M}^G(C,X,d) \to\ol{\M}^{G,\quot}(C,X,d).\end{equation}
For each fixed bundle this map collapses bubbles of the section $u$
and replaces them with base points with multiplicity given by the
degree of the bubble tree.  Since the Givental morphism
\eqref{givmor}, the quot-scheme compactification
$\ol{\M}_n^{G,\quot}(C,X,d)$ and the forgetful morphism
\(\ol{\M}^G_n(C,X,d) \to\ol{\M}^{G}(C,X,d)\)
are proper, $\ol{\M}^G_n(C,X,d)$ is
proper.


\section{Applications}

\subsection{Presentations of quantum cohomology and quantum K-theory}

The first group of applications use that the linearized quantum Kirwan
map is a ring homomorphism.  In good cases, such as the toric case,
one can prove that the linearized quantum Kirwan map is a surjection
$$ D_\alpha \kappa_X^G : T_\alpha QH_G(X) \to T_{\kappa_X^G(\alpha)}
QH(X \qu G) $$
and so obtain a presentation for the quantum product in quantum
cohomology \cite{gw:surject} or quantum K-theory at 
$\kappa_X^G(\alpha)$.  A simple
example is projective space itself, which is the git quotient of a
vector space by scalar multiplication of $X = \C^k$ by $G =
\C^\times$:
$$ \P^{k-1} = \C^k  \qu \C^\times $$
We show how to derive the relation in quantum cohomology or quantum
K-theory.  The moduli space of affine gauged maps of class $d \in
H_2^G(X) \cong \Z$ is the space of $k$-tuples of degree $d$
polynomials
$$ u(z) = (u_1(z),\ldots, u_k(z)) $$
with non-zero leading order term given by the top derivative
$u^{(d)}$.  Thus its evaluation at infinity lies on the
semistable locus.
$$ \M^G_{1,1}(\bA, X) \cong (\C^{k(d+1)} - (\C^{kd} \times \{ 0 \} ))/
\C^\times .$$
Let $\alpha = 0$ and let $\beta \in H_2^G(X)$ be the Euler class of
the trivial vector bundle $X \times \C \to X$ where
\(G\) acts on \(\C\) with weight one.  The pull-back 
$$\ev_1^*(
(X \times \C \to X)^{\oplus d}) \to \ol{\M}^G_{1,1}(\bA,X,d) $$
has a canonical section given by the first $d$ derivatives,
$$\sigma: \ol{\M}^G_{1,1}(\bA,X,d)\to \ev_1^* ((X \times \C \to X)^{\oplus d}), \quad [u] \mapsto
[u^{(i)}(z_1)]_{i < d} . $$
One can check that this section extends to a section over the
compactification $\ol{\M}^G_{1,1}(\bA,X,d)$ which is non-vanishing on the
boundary.  The map $\ev_{\infty,d}: \sigma^{-1}(0) \to X \qu G$ is an
isomorphism, since the evaluation takes the terms of top order.  This
implies
$$ (\ev_{\infty,d,*} \ev_1^*) \beta^{dk} = \ev_{\infty,d,*}
|_{\sigma^{-1}(0)} 1 = 1 .$$
In particular for $ d= 1$ the image of $\beta^k$ under the linearized
quantum Kirwan map is
$$ D_0 \kappa_X^G(\beta^{k}) = q .$$
In quantum cohomology one obtains
$$ QH(\P^{k-1}) = \Q[\beta,q]/ (\beta^k - q). $$
This was one of the first computations in quantum cohomology but the
advantage of this approach is that it works equally well in quantum
K-theory.  
In this case $\beta$ is the exterior algebra on the dual of $L$,
$\beta = 1 - L^{-1} \in K(\P^{k-1})$, where $L$ is the class of the
hyperplane line bundle.  One obtains
$$ QK(\P^{k-1}) = \Q[L,L^{-1},q]/ ( (1 - L^{-1})^k - q) .$$
More generally, one gets Batyrev-type presentations of the quantum
cohomology and quantum K-theory of toric smooth Deligne-Mumford stacks
with projective coarse moduli spaces.

\subsection{Solutions to quantum differential equations}

A second group of applications concerns explicit formulas for
solutions to quantum differential equations.  Recall that a
fundamental solution in quantum cohomology is given by one-point
descendent potentials (J-function)
$$ \tau_{X \qu G}^\pm: QH(X \qu G) \to QH(X \qu
G)[\zeta,\zeta]^{-1}]] $$
which appear when one factorizes the genus zero graph
potential $\tau_{X \qu G}$ for $C = \P^1$ into
contributions from $\C^\times$-fixed points at $0,\infty$;
there is a parallel discussion in quantum K-theory.  One has a
similar factorization of the gauged potential into contributions from
$0, \infty$ in terms of {\em localized gauged potentials}
$$ \tau_{X,G}^\pm: QH_G(X) \to QH(X \qu G)[[\zeta]] .$$
Using the fact that the cobordism given by scaled gauged maps is
equivariant, one obtains the following ``localized'' version of the
adiabatic limit theorem:
$$ \tau_{X,G}^{\pm} = \tau_{X \qu G}^\pm \circ \kappa_X^G
.$$
Often, the localized gauged potentials are easier to compute than the
localized graph potentials.  For example, for the case of a vector
space $X$ with a representation of a torus $G$ one obtains at $\alpha
= 0 $
$$ \tau_{X,G}^{\pm}(0)  = \sum_{d \in H_2^G(X)} q^d \frac{
  \prod_{j=1}^k \prod_{m = -\infty}^{0} (1 - X_j^{-1} \zeta^m ) }{
  \prod_{j=1}^k \prod_{m=-\infty}^{\mu_j(d)} (1 - X_j^{-1} \zeta^m )
}$$
where $X_j$ is the class of the line bundle given by the $j$-th weight
space of $X$.  In this case the relationship between the two
potentials has essentially been established in a series of papers by
Givental \cite{gi:pe1}, \cite{gi:pe2}, \cite{gi:pe3}, \cite{gi:pe4},
\cite{gi:pe5}, \cite{gi:pe6}.  This gives yet another way of producing
relations in the quantum cohomology or quantum K-theory.

For actions of non-abelian groups, one may also obtain formulas for
fundamental solutions to the quantum differential equation by an {\em
  abelianization} procedure suggested by Martin \cite{mar:sy} in the
classical case and by Bertram-Ciocan-Fontanine-Kim \cite{be:qu} in the
quantum case.  In the case of gauged maps, a proof in K-theory can be
reduced to Martin's case by taking the {\em small linearization}
chamber in which the bundle part of a gauged genus zero map is
automatically trivial.  As a result, one also gets formulas for
fundamental solutions to quantum differential equations for such
spaces as Grassmannians and moduli of framed rank $r$
sheaves, charge \(k\) (in the
Atiyah-Hitchin-Drinfeld-Manin description).  For example for framed
sheaves one obtains the formula for the fundamental solution (with
equivariant parameters $\xi_1,\xi_2$ corresponding to the action of
the torus on the plane):
\begin{multline} \tau_{X \qu G,-}(0) = \sum_{d \ge 1} q^d
  \sum_{\ul{d}: d_1 + \ldots + d_k = d} \prod_{i \neq j}
  \frac{\Delta_{\ul{d}}(\theta_i - \theta_j, \xi_1 + \xi_2)
    \Delta_{\ul{d}}(\theta_i -\theta_j, 0 )} {\Delta_{\ul{d}}(\theta_i
    - \theta_j,\xi_1) \Delta_{\ul{d}}(\theta_i - \theta_j,
  \xi_2)} \\
  \prod_{i =1 }^k \Delta_{\ul{d}}(\theta_i,0)^{-r}
  \Delta_{\ul{d}}(-\theta_i, \xi_1 + \xi_2)^{-r} .\end{multline}
where $ \Delta_{\ul{d}}(\theta, w) := ( \prod_{l = -\infty}^{\theta \cdot
  \ul{d}} (\theta + w+ l \zeta))/( \prod_{l = -\infty}^{0}
  (\theta + w+l
\zeta)) $ and $\theta_1,\ldots,\theta_k$ are the Chern roots of the
tautological bundle, \(\theta=\sum k_i \theta_i\) any
integral linear combination and \(\theta\cdot \ul{d}=\sum
d_i k_i\).  A version of this formula for $r = 1$ (with a
``mirror map'' correction) is proved by Ciocan-Fontanine-Kim-Maulik.
The $r > 1$ formula seems to be new.  Furthermore, we claim
that same approach works in quantum K-theory.  For the Grassmannian, one obtains
a formula for the descendent potential in both quantum cohomology and
quantum K-theory.  From this one may deduce the standard presentation
of the quantum cohomology or quantum K-theory of the Grassmannian as a
quotient of the ring of symmetric functions as in Gorbounov-Korff
\cite{gorb:int} and Buch-Chaput-Mihalcea-Perrin \cite{buch:kchev}.

\subsection{Wall-crossing for Gromov-Witten invariants}

A final application is a conceptual result on the relationship between
Gromov-Witten invariants of symplectic quotients.  For simplicity let
$G = \C^\times$ and let $X$ be a smooth projective $G$-variety
equipped with a family of linearizations $\ti{X}_t$.  Associated to
the family of linearizations we have a family of git quotients $ X
\qu_t G$; we suppose that stable $\neq$ semistable at $t = 0$ in which
case $X \qu_t G$ undergoes a flip.  For example, if $X = \P^4$ with
weights $0,1,1,-1,-1$ then $X \qu_t G$ undergoes an Atiyah flop as $t$
passes through zero.  Using $\C^\times$-equivariant versions of the
gauged Gromov-Witten invariants we obtain an explicit formula for the
difference
\begin{equation} \label{diff}
 \tau_{X \qu_+ G} \kappa_{X,+}^G -
  \tau_{X \qu_- G} \kappa_{X,-}^G \end{equation}
in either cohomology or K-theory.  We say that the variation is {\em
  crepant} if 
the weights at the $G$-fixed points at $t = 0$ sum to zero; in this
case the birational equivalence between $X \qu_- G$ and $X \qu_+ G$ is
a combination of flops.  In this case one can show that the difference
in \eqref{diff} is, as a distribution in $q$, vanishing almost
everywhere:
$$ 
\tau_{X \qu_+ G} \kappa_{X,+}^G =_{\text{a.e. in q}} \tau_{X \qu_- G}
\kappa_{X,-}^G $$
in both quantum cohomology and K-theory as well. In the
later this holds  after shifting by the square
root of the canonical line bundle.  This is a version of the
so called {\em crepant transformation} conjecture in \cite{cr:crep}, with the added benefit
that the proof is essentially the same for both cohomology (covered in
\cite{wall}) and K-theory.

\def\cprime{$'$} \def\cprime{$'$} \def\cprime{$'$} \def\cprime{$'$}
\def\cprime{$'$} \def\cprime{$'$}
\def\polhk#1{\setbox0=\hbox{#1}{\ooalign{\hidewidth
      \lower1.5ex\hbox{`}\hidewidth\crcr\unhbox0}}} \def\cprime{$'$}
\def\cprime{$'$}

\end{document}